\UseRawInputEncoding
\documentclass[12pt]{amsart}
\usepackage{epsfig, color, amsmath, esint, hyperref, mathrsfs, bm, mathtools, comment, amsfonts, amssymb, enumerate}
\usepackage[normalem]{ulem} 
\usepackage[dvipsnames]{xcolor}
\usepackage[normalem]{ulem}

\hypersetup{colorlinks}
\hypersetup{citecolor=blue}
\hypersetup{urlcolor=blue}
\makeatother

\theoremstyle{definition}
\def\fnum{equation} 
\newtheorem{theorem}[\fnum]{Theorem}
\newtheorem{corollary}[\fnum]{Corollary}

\newtheorem{lemma}[\fnum]{Lemma}
\newtheorem{conjecture}[\fnum]{Conjecture}
\newtheorem{definition}[\fnum]{Definition}
\newtheorem{example}[\fnum]{Example}
\newtheorem{remark}[\fnum]{Remark}
\newtheorem{proposition}[\fnum]{Proposition}
\newtheorem{mainthm}{Theorem}

\numberwithin{equation}{section}

\newcommand{\Ker}{{\text{Ker}}}
\newcommand{\rank}{{\text{rank}}}
\newcommand{\inj}{{\text{inj}}}

\newcommand{\diam}{{\text {diam}}}

\newcommand{\supp}{{\text {supp}}}

\newcommand{\bet}{\text{b}_1}

\newcommand{\Id}{\text{Id}}
\newcommand{\iso}{\text{Iso}}
\newcommand{\vv}{\text{v}}
\newcommand{\ww}{\text{w}}
\newcommand{\dd}{\mathsf d}

\newcommand{\thi}{\text{thick}}

\newcommand{\sobo}{H^{1,2}}

\newcommand{\rcd}{\text{RCD}}
\newcommand{\prob}{\text{Prob}}
\newcommand{\cdkn}{\text{CD}(K,N)}
\newcommand{\haus}{\mathcal{H}}

\newcommand{\RCD}{\text{RCD}}

\newcommand{\Ric}{\text{Ric}}
\newcommand{\mm}{\mathfrak{m}}
\newcommand{\R}{\mathbb{R}}
\newcommand{\Z}{\mathbb{Z}}

\title{Topological rigidity of small RCD(K,N) spaces with maximal rank}
\author{Sergio Zamora}\address{Oregon State University}
\author{Xingyu Zhu} \address{Michigan State University}

\begin{document}

\begin{abstract}
For a polycyclic group $\Lambda$,  $\rank (\Lambda )$ is defined as the number of $\mathbb{Z}$ factors in a polycyclic decomposition of $\Lambda$. For a finitely generated group $G$,  $\rank (G)$ is defined as the infimum of $ \rank (\Lambda )$ among finite index polycyclic subgroups $\Lambda \leq G$.

For a compact $\rcd (K,N)$ space $(X,\mathsf{d}, \mm)$ with $\diam (X) \leq \varepsilon (K,N)$,  the rank of $\pi_1(X)$ is at most $N$. We show that in case of equality, $X$ is homeomorphic to an infranilmanifold, generalizing a result by Kapovitch--Wilking to the non-smooth setting.  
\end{abstract}

\maketitle




\section{Introduction}

A classical result of Gromov \cite{gromov}, refined by Ruh \cite{ruh}, asserts that a closed $N$-dimensional Riemannian manifold $X$ with 
\[          \hspace{0.8cm}           \vert \sec (X) \vert \leq K , \hspace{3cm} \diam (X) \leq \varepsilon (K,N),  \]
is diffeomorphic to an infranilmanifold.   Recall that an \emph{infranilmanifold} is  a compact smooth manifold of the form $G / \Gamma$ with  $G$ a simply-connected nilpotent Lie group and $\Gamma$ a torsion-free discrete subgroup of the semi-direct product of $G$ with a compact group of automorphisms of $G$. If $\Gamma \leq G$, then the infranilmanifold is called a \textit{nilmanifold}.

This full topological control is lost if one works with only a Ricci curvature bound; this is illustrated by the fact that the topology of Ricci flat manifolds is still quite mysterious. However, the fundamental group is very well understood. By the work of Kapovitch--Wilking \cite{kapovitch-wilking}, if a closed $N$-dimensional Riemannian manifold  $X$ satisfies
\[          \hspace{0.9cm}           \Ric (X)  \geq K , \hspace{3cm} \diam (X) \leq \varepsilon (K,N) , \]
then $\pi_1(X)$ contains a subgroup $\Lambda $ of index $[\pi_1(X): \Lambda ] \leq C ( K , N)$ generated by elements $\{ u_1, \ldots , u_N \} \subset \Lambda $ with $ [u_i , u_j ] \in \langle u_{j+1}, \ldots , u_{N} \rangle $ for each $i,j \in \{ 1, \ldots , N \}$. Moreover, if 
\[    \langle u_j, \ldots , u_N \rangle / \langle u_{j+1}, \ldots  , u_{N} \rangle \cong \mathbb{Z} \text{ for all } j \in \{ 1, \ldots , N \},  \]
then $X$ is homeomorphic to an infranilmanifold. Our main result is a generalization of this maximal rank topological rigidity to the non-smooth setting.  We refer the reader to Section \ref{subsec:rank} for the definition of rank of a finitely generated group.  

\begin {mainthm}\label{thm:max-rank-nilmanifold}
 For each $K \in \mathbb{R},$ $N \geq 1  $, there is $\varepsilon > 0$ such that if $(X,\mathsf{d} ,  \mm ) $ is an $\rcd (K,N)$ space with $\diam (X) \leq \varepsilon $, then $ \rank (\pi_1(X))  \leq N $, and in case of equality, $X$ is homeomorphic to an infranilmanifold of dimension $N$. 
\end{mainthm}

Theorem \ref{thm:max-rank-nilmanifold} can be regarded as a non-abelian analogue of the following result \cite[Theorem 1.2.3]{mondello-mondino-perales}. Recall that for a path connected topological space $X$, its first Betti number $\bet (X)$  can be  defined as the supremum of $k $ for which there is a surjective morphism $\pi_1(X) \to \mathbb{Z}^k$.

\begin{theorem}[Mondello--Mondino--Perales]\label{thm:mmp}
For each $K \in \mathbb{R},$ $N \geq 1 $, there is $\varepsilon > 0$ such that if $(X,\mathsf{d},  \mm ) $ is an $\rcd (K,N)$ space with $\diam (X) \leq \varepsilon $, then $\bet (X) \leq N $, and in case of equality, $X$ is  bi-H\"older homeomorphic to a flat torus of dimension $N$.
\end{theorem}

Previoulsy,  Cheeger--Colding  established that under the hypotheses of Theorem \ref{thm:mmp}, if $X$ is a smooth Riemannian manifold, then it is diffeomorphic to a flat torus \cite[Theorem A.1.13]{cheeger-colding-i}.  As a second goal of this paper,  we provide some missing details in the proofs of  \cite[Theorem 1.2.3]{mondello-mondino-perales} and \cite[Theorem A.1.13]{cheeger-colding-i}.  We first indicate where the missing details are. \color{black}By contradiction, one assumes there is a sequence of $\rcd (K,N)$ spaces $(X_i, \mathsf{d}_i, \mm_i)$ with 
\[ \diam (X_i) \to 0,  \hspace{3cm}  \bet (X_i) = N , \]
but $X_i$ not bi-H\"older homeomorphic to a flat torus. In \cite{mondello-mondino-perales}, a sequence of finite sheeted covers $X_i^{\prime} \to X_i$ is constructed in such a way that $X_i^{\prime}$ converges in the Gromov--Hausdorff sense to the flat torus $ \mathbb{T}^N$. From here, the Cheeger--Colding version of the Reifenberg Theorem  \cite[Appendix A]{cheeger-colding-i} implies that $X_i^{\prime}$ is bi-H\"older homeomorphic to $\mathbb{T}^N$ for $i$ large enough, and it is deduced using topological arguments that the spaces $X_i$ are homeomorphic to $\mathbb{T}^N$.

Their argument concludes here, but the proof is not finished, since a bi-H\"older homeomorphism $ X_i^{\prime} \to \mathbb{T}^N$ doesn't necessarily descend to a bi-H\"older homeomorphism between $X_i$ and a flat torus. Moreover, if $N\geq 5 $, there exists exotic tori, i.e., smooth manifolds homeomorphic to the standard $N$-torus but not diffeomorphic to it \cite[Section 15.A]{wall}, and any such 
exotic tori admits a finite-sheeted covering diffeomorphic to the standard $N$-torus \cite[Theorem A.1]{fisher-kalinin-spazier}.

\color{black}

We address this with the following result, which is a modification of the \emph{canonical homeomorphism} constructed by Honda--Peng in \cite[Theorem 6.3]{honda-peng} (see the end of Section \ref{sec:canonical} for further details).  We refer the reader to Section \ref{sec:mgh} for the definition of measured Gromov--Hausdorff convergence. 


\color{black}

\begin{mainthm}\label{thm:equivariant-torus}
    For each $K \in \mathbb{R}$, $N \geq 1$, $\varepsilon > 0 $, there is $\delta > 0 $ such that if  $(X, \mathsf{d}, \mm  )$ is a compact $\rcd (K,N)$ space $\delta$-close to $\mathbb{T}^N$ in the measured Gromov--Hausdorff sense,  and  there is an abelian group $G \leq \iso (X)$ with $\diam (X / G ) < \delta $, then there is a free isometric action of $G$ on $\mathbb{T}^N$ and a $G$-equivariant homeomorphism $F : X \to \mathbb{T}^N$ with 
    \begin{equation}\label{eq:bi-holder-good}
         (1 - \varepsilon ) \, \mathsf{d}( x,y ) ^{1 + \varepsilon } \leq \mathsf{d}( F (x), F (y)  ) \leq (1 + \varepsilon ) \, \mathsf{d}(x,y) \, \text{ for all } x,y \in X.  
    \end{equation} 
    Furthermore, if $X$ is a smooth Riemannian manifold, the map $F$ can be taken to be a diffeomorphism.     
\end{mainthm}

    In the setting of smooth Riemannian manifolds, Huang obtained a similar result for more general limit spaces and groups  \cite[Corollary 1.9]{huang}, which can be used in an analogous way to fill the gap in the proof of  \cite[Theorem A.1.13]{cheeger-colding-i}.

\subsection{Strategy and Outline}\label{subsec:outline}

The Borel conjecture asserts that any two closed aspherical manifolds with isomorphic fundamental group are homeomorphic \cite{rosenberg}. Since this conjecture is known to be true for manifolds with virtually nilpotent fundamental groups (see Theorem \ref{thm:borel-nilpotent}), and the Margulis Lemma holds for $\rcd (K,N)$ spaces (see \cite{deng-santos-zamora-zhao}), then the proof of Theorem \ref{thm:max-rank-nilmanifold} consists mainly of showing that if $\diam (X) \leq \varepsilon $ and $\rank (\pi_1(X)) = N$, then $X$ is a closed aspherical topological manifold.  This approach is also used in \cite{kapovitch-wilking} for the  proof of the smooth case. However, our proof is significantly different and builds upon the work of Breuillard--Green--Tao on the structure of approximate groups \cite{breuillard-green-tao}.

Theorem \ref{thm:equivariant-torus} is proven by contradiction. Assuming there is a sequence of $\rcd (K,N)$ spaces $(X_i, \mathsf{d}_i, \mm_i)$ and groups $G_i \leq \iso (X_i)$ that contradict the statement of the theorem, we show that for $i$ large enough, using eigenfunctions of the Laplacian on $X_i$, one can find  representations  $G_i \to O(2N) $ and equivariant maps $\Phi_i : X_i \to \mathbb{R}^{2N}$ such that $\mathbb{T}^N \subset \mathbb{R}^{2N}$ is $G_i$-invariant and the maps $\Phi_i$ satisfy the conditions of \cite[Theorem 1.2.2]{honda-peng}, which we apply to deduce \eqref{eq:bi-holder-good}.

In Section \ref{sec:discussion} we discuss related results and conjectures. In Section \ref{sec:prelims} we cover the background material we need for our main results. In  Section \ref{sec:topology} we prove the technical topological tools required for our main theorems. In Section \ref{sec:aspherical} we prove Theorem \ref{thm:max-rank-nilmanifold}.   In Section \ref{sec:canonical} we prove Theorem \ref{thm:equivariant-torus} and  fill in the details  \color{black}in the proof of Theorem \ref{thm:mmp}.

\section{Related problems and results}\label{sec:discussion}

Under the conditions of Theorem \ref{thm:max-rank-nilmanifold}, if $X$ is a smooth $N$-dimensional Riemannian manifold with $\rank (\pi_1(X)) = N$, then it follows from \cite[Proposition 5.4]{naber-zhang} and the work of Huang--Kong--Rong--Xu \cite{huang-kong-rong-xu} that $X$ is diffeomorphic to an infranilmanifold (this was also proven by Naber--Zhang \cite{naber-zhang} assuming an upper Ricci curvature bound). We note, however, that their proof relies on the Ricci flow, so their techniques do not extend to the non-smooth setting. Later Rong \cite{rong} gave a more geometric proof inspired by Cheeger--Fukaya--Gromov collapsing theory \cite{cheeger-fukaya-gromov}, but some tools used there still rely on the smooth structure. In view of this discussion and Theorem \ref{thm:mmp}, it would be interesting to know if the conclusion of  Theorem \ref{thm:max-rank-nilmanifold} can be improved to a bi-H\"older homeomorphism. 

\begin{conjecture}\label{con:bi-holder}
   For each $K \in \mathbb{R},$ $ N\geq 1  $, there is $\varepsilon > 0$ such that if $(X,\mathsf{d} ,  \mm ) $ is an $\rcd (K,N)$ space with $\diam (X) \leq \varepsilon $, then $ \rank (\pi_1(X))  \leq N $, and in case of equality, $X$ is bi-H\"older homeomorphic to an infranilmanifold of dimension $N$.  
\end{conjecture}

The following proposition shows that the approach of constructing finite-sheeted covers $X_i^{\prime} \to X_i$ that don't collapse fails when dealing with Conjecture  \ref{con:bi-holder} because in general, fundamental groups of infranilmanifolds are not virtually abelian.

\begin{proposition}\label{pro:failure-finite-sheeted-covers}
 Let $X_i$ be a sequence of closed  $N$-dimensional Riemannian manifolds with  $ \Ric (X_i) \geq - \frac{1}{i} $ and $\pi_1(X_i)$ torsion-free for all $i$. If there is a sequence of finite-sheeted covers $X_i^{\prime} \to X_i$ with 
\[  \diam (X_i^{\prime}) \leq D, \hspace{3cm} \haus^N (X_i^{\prime}) \geq \nu   \]
for some $D, \nu > 0$, then $\pi_1(X_i)$ is virtually abelian for large $i$.  
\end{proposition}
\begin{proof}
Arguing by contradiction, after passing to a subsequence, we can assume $\pi_1(X_i)$ is not virtually abelian for any $i$. Let $\tilde{X}_i$ be the sequence of universal covers, $p_i \in X_i^{\prime}$,  and $\tilde{p}_i \in \tilde{X}_i$ in the preimage of $p_i$.  By Theorems \ref{thm:fukaya-yamaguchi} and \ref{thm:fukaya}, after further passing to a subsequence, the triples $(\tilde{X}_i, \tilde{p}_i, \pi_1(X_i^{\prime}))$ converge in the equivariant Gromov--Hausdorff sense to a triple $(Y, \tilde{p} , G)$, while $X_i^{\prime}$ converges to $Y / G$. 

Since the sequence $X_i^{\prime}$ doesn't collapse, neither does the sequence $\tilde{X}_i$ and the limit group $G$ is discrete so by the work of Pan--Rong 
\cite[Theorem 2.17]{pan-rong} the groups $\pi_1(X_i^{\prime})$ are isomorphic to $G$ for large $i$. By the work of Pan--Wang \cite[Theorem 1.5]{pan-wang}, $\pi_1(Y/ G)$ is the quotient of $G$ by a subgroup generated by torsion elements, but since $G$ is torsion-free such subgroup must be trivial and $\pi _1(Y/ G) = G$. 

Finally,  $Y/G$ equipped with $\mathcal{H}^N$ is a compact $\rcd (0, N)$ space so its fundamental group is virtually abelian, hence so is $\pi_1(X_i)$ for large $i$; a contradiction. 
\end{proof}

If $\sec ( X_i ) \geq -1 / i$, one could remove from Proposition \ref{pro:failure-finite-sheeted-covers}  the hypothesis that the groups $\pi_1(X_i)$ are torsion-free since one could use Perelman stability theorem \cite{kapovitch-perelman, perelman} to  deduce  that $X_i^{\prime} $ and $Y / G$ are homeomorphic for large $i$.   Example \ref{exa:heis} below illustrates the intuition behind Proposition \ref{pro:failure-finite-sheeted-covers}: in order to obtain finite-sheeted covers with controlled volume, one must ``unwind'' the manifold in all directions, which can lead to the diameter growing to infiniy if the groups are not abelian.

\begin{example}\label{exa:heis}
Consider the $3$-dimensional Heisenberg groups
\[
H^3(\R)=\left\{ \left. \begin{bmatrix}
    1 & x & z\\
    0 & 1 & y\\
    0 & 0 & 1
\end{bmatrix} \right| x,y,z \in \R \right \}, \quad H^3(\Z)=\left\{\left. \begin{bmatrix}
    1 & a & c\\
    0 & 1 & b\\
    0 & 0 & 1
\end{bmatrix}\right| a,b,c\in \Z \right\}.
\]
Let $\{ \vv _1, \, \vv _2, \, \vv _3 \} $ be the basis of the Lie algebra of $H^3(\R)$ given by
\[
\vv_1= \begin{bmatrix}
    0 & 1 & 0\\
    0 & 0 & 0\\
    0 & 0 & 0
\end{bmatrix} ,
\hspace{2cm}
\vv_2= \begin{bmatrix}
    0 & 0 & 0\\
    0 & 0 & 1\\
    0 & 0 & 0
\end{bmatrix} ,
\hspace{2cm}
\vv_3= \begin{bmatrix}
    0 & 0 & 1\\
    0 & 0 & 0\\
    0 & 0 & 0
\end{bmatrix}.
\]
For $ i  \in \mathbb{N}$, consider the left invariant Riemannian metric $g_i$ on $H^3 (\mathbb{R})$ for which
\begin{gather*}
    \Vert \vv _1 \Vert _{g_i} ^4  = \Vert \vv _2 \Vert _{g_i} ^4 = \Vert \vv _3 \Vert _{g_i}  = 1/ i ^4 , \hspace{1.7cm}     \langle \vv _{j_1} , \vv_{j_2} \rangle _{g_i} = 0 \text{ for } j_1 \neq j_2 ,  \hspace{0.5cm} 
\end{gather*}   
and let $X_i$ be the quotient of $(H^3(\mathbb{R}), g_i)$ by the action of $H^3(\Z)$. Then 
\[      \diam (X _i) \to 0 , \hspace{3cm}  |\sec ( X_i)  | \to 0 . \] 
Note that $\rank ( \pi _1 (X_i)) = \rank (H^3(\Z))=3$.   If one wants to construct finite-sheeted covering spaces $X_i ^{\prime } \to X_i $ with 
\[ \diam (X_i^{\prime}) \leq D , \hspace{3cm} \haus ^N (X_i^{\prime}) \geq \nu  \]
for some $D, \nu > 0 $, then one would need to unwind in all directions. That is, if $u_j : =  \exp (\vv_j) \in H^3(\Z )$ for $j \in \{ 1, 2, 3 \} $ and $X_i^{\prime} = ( H^3 (\R) , g_i ) / \Gamma_i$ for some subgroup $\Gamma_i \leq H^3(\Z )$, then there is a sequence $N_i \in \mathbb{N}$ going to infinity at least as fast as $i$ for which 
\[   u_1 ^{n_1} u_2 ^{n_2}u_3^{n_3}  \notin \Gamma _ i  \]
for all $n_1, n_2, n_3 \in \Z $ with $\vert n_1 \vert , \vert n_2 \vert  \leq N_i $, $ \vert n_3 \vert  \leq N_i ^ 4$. The projection  $\rho : H^3 (\Z ) \to  \Z^2$ given by
\[        u_1 \mapsto (1,0) ,  \hspace{1.5cm} u_2 \mapsto (0,1) , \hspace{1.5cm} u_3 \mapsto (0,0)           \]
sends $\Gamma _ i $ to a lattice $\Lambda_i $ generated by a basis $\{  \ww _{1,i} ,  \ww_{2,i} \}$. If we denote by $A_i$ the area of the parallelogram spanned by $ \ww _{1,i} $ and $\ww _{2,i}$, then  $A_i \geq N_i^4$. This follows from the fact that if $g_{j,i} \in \Gamma _i \cap \rho ^{-1} (\ww _{j,i})$, then $[ g_{1,i}, g_{2,i}] = u_3 ^{A_i} \in \Gamma _ i $. From here, it is not hard to show that  
\[ \diam ( \mathbb{R}^2 / \Lambda _i ) \geq  \sqrt{A_i } / 10 \geq N_i ^2 / 10  . \]  
Finally, notice that the map $\overline {\rho } : H^3 (\R ) \to \mathbb{R}^2$ given by 
\[    \overline{\rho} \left(    \begin{bmatrix}
    1 & x & z\\
    0 & 1 & y\\
    0 & 0 & 1
\end{bmatrix}   \right) : = (x,y)                          \]
induces an $i$-Lipschitz surjective map $X_i ^{\prime} \to \mathbb{R}^2 / \Lambda _i$. For $i$ large enough this implies
\[     \diam (X_i^{\prime})  \geq \diam (\mathbb{R} ^2 / \Lambda _i) / i   \geq  N_i ^2 / 10 i  \geq \sqrt{i}  \to \infty ,  \]
which is a contradiction.
\end{example}

\begin{remark}
    The missing piece to prove Conjecture \ref{con:bi-holder} is Lemma \ref{lem:local-to-global-approximation}. That is, to prove the conjecture it would be enough to show that if in addition to the conditions of such lemma, one has
    \begin{itemize}
        \item $X$ is a non-collapsed $\rcd (K,N)$ space.
        \item $Y$ is an almost flat nilpotent Lie group of dimension $N$.
        \item $f$ satisfies the estimate
        \[              (1 - \delta  ) \, d( x,y ) ^{1 + \delta } \leq d ( f (x), f (y)  ) \leq (1 + \delta ) \, d (x,y) \, \text{ for all } x,y \in B_{10} (p)      .  \]
        \item  $G_Y$ acts freely on $Y$. 
    \end{itemize}
    then the produced $\tilde{f}$ can be chosen to be locally bi-H\"older (and not just a homotopy equivalence).  Arguing as in \cite[Theorem 3.3]{kapovitch-mondino} one can construct a global homeomorphism $\tilde{f} : X \to Y$ that is locally bi-H\"older, but making this function equivariant seems to be a very interesting technical challenge. 
    
     In principle,  one would need to choose the bump functions $\rho _g$ more carefully, but this may not be enough as one may need to also modify the function $f$  as it is done in the smooth setting by Huang in \cite{huang}. Moreover, the exact same construction is possible as the groups we are dealing with are amenable, making the equivariant map constructed in \cite[Theorems 1.3, 1.8]{huang} the perfect candidate to be the desired map.
\end{remark}

\subsection{Collapsed $\rcd (K,N)$ spaces} For an $\rcd (K,N)$ space $(X,\mathsf{d},\mm )$ with $\diam (X) \leq \varepsilon (K,N)$, it follows from the proof of Theorem \ref{thm:max-rank-nilmanifold} that $\rank (\pi_1 (X))$ is at most the rectifiable dimension of $X$, which is an integer bounded above by $N$. One can ask whether the rigidity of Theorem \ref{thm:max-rank-nilmanifold} can be pushed to when $\rank (\pi_1(X))$ equals the rectifiable dimension of $X$. This is not possible due to \cite{zhou} (see also \cite{hupp-naber-wang}). 

\begin{theorem}[Zhou]
For each $n \geq 3$, $\varepsilon > 0 $, there is an $\rcd (-1, n + 2 )$ space $(X,\mathsf{d}, \mm )$ of rectifiable dimension equal to $n$ with
\[    \diam (X) \leq \varepsilon , \hspace{3cm}  \pi_1(X)= \Z^n,  \] 
but no point of $X$ admits an open neighborhood homeomorphic to $\mathbb{R}^n$. 
\end{theorem} 
    \begin{proof} This follows by performing the main construction of \cite{zhou} to a flat torus of dimension $n$ and diameter $ \leq \varepsilon $.  \end{proof}

It is however possible that if the dimension gap is small enough, the topological rigidity still holds.

\begin{conjecture}
   For each $K \in \mathbb{R}$, $N \in \mathbb{N}$, there is $\varepsilon > 0 $ such that if $(X, \mathsf{d}, \mm )$ is an $\rcd (K, N + \varepsilon )$ space with $ \diam (X) \leq \varepsilon $ and $ \rank (\pi_1(X)) = N $, then $X$ is homeomorphic to an infranilmanifold.
\end{conjecture}

\subsection{Fibrations} Theorem \ref{thm:max-rank-nilmanifold} is a step towards the following conjectures which would extend to $\rcd (K,N)$ spaces work of Naber--Zhang \cite{naber-zhang},   Huang--Kong--Rong--Xu \cite{huang-kong-rong-xu}, and Huang--Wang \cite{huang-wang}.

\begin{definition}
    Let $X$ be a geodesic space, $x \in X$, and $\varepsilon > 0$. The \emph{fibered fundamental group } $\Gamma_{\varepsilon} (x)$ is defined to be the image of the map  $     \pi_1 (B_{\varepsilon}(x)) \to \pi_1(B_1(x))          $ induced by the inclusion. 
\end{definition}

\begin{conjecture}\label{con:wang}
    There is $\varepsilon ( K, N ) > 0 $ such that the following holds. Let $X$ be a closed $m$-dimensional Riemannian manifold with 
    \[     \vert \sec (X) \vert \leq 1 ,\hspace{3cm} \inj (X) \geq 1  ,    \]
    and $(X_i, \mathsf{d}_i, \mm_i )$ a sequence of $\rcd (K,N)$ spaces converging in the Gromov--Hausdorff sense to $X$. Assume that for all $i$ large enough and any $x \in X_i$ one has 
    \[    \rank ( \Gamma _{\varepsilon} (x)) = N - m     .         \]
    Then for large enough $i$ there are continuous Gromov--Hausdorff approximations $X_i \to X$ that are locally trivial fibrations with infranilmanifold fiber.  
\end{conjecture}

\begin{conjecture}
    Let $(X_i, \mathsf{d}_i, \mm_i)$ be a sequence of compact $\rcd(K,N)$ spaces that converges in the measured Gromov--Hausdorff sense to a closed $m$-dimensional Riemannian manifold $X$. Assume that for all $i$ one has
    \[  \bet (X_i) - \bet (X) = N - m . \]
    Then  for large enough $i$ there are continuous Gromov--Hausdorff approximations $X_i \to X$ that are also locally trivial fibrations with fiber $\mathbb{T}^{N - m}$.
\end{conjecture}

\begin{remark}
After the initial version of this paper appeared on arXiv, Wang \cite{wang-bounded-covering} proved both Conjecture \ref{con:bi-holder} and Conjecture \ref{con:wang} by building upon the techniques developed in this paper and in \cite{zamora-limits}. Moreover, while we pointed out in Proposition \ref{pro:failure-finite-sheeted-covers} that constructing non-collapsing finite-sheeted covers would not work in that setting, Wang found a way to overcome this difficulty and successfully implemented this technique to prove both conjectures.

\end{remark}

\section{Preliminary material}\label{sec:prelims}

\subsection{Notation}  For a metric space $X$, we denote its group of isometries by $\iso (X)$. For $p \in X$, $r > 0$, we denote the open ball of radius $r$ around $p$ by $B_r^X(p)$ or $B_r(p)$, depending on whether we want to emphasize the metric space used.

For $N \geq 0 $, we denote by $\mathcal{H}^N$ the $N$-dimensional Hausdorff measure.

For $N \in \mathbb{N}$, we denote by 
$ \mathbb{T}^N$ the torus 
\[\{ (x_1, y_1, \ldots , x_N, y_N) \in \mathbb{R}^{2N}  \, \vert  \,   x_j ^2  + y_j ^2 = 1 \text{ for each } j  \}\]
equipped with its standard flat metric and the $N$-dimensional Hausdorff measure.

$H_m (\cdot )$ denotes singular homology with integer coefficients. Singular simplices, singular chains, and the boundary operator $\partial$ are as defined in \cite[Chapter 2]{hatcher}.

For a vector space $V$ and a subset $S \subset V$, we denote by $\langle S \rangle \leq V$ the vector subspace generated by $S$. For $\vv \in V$, we also denote $\langle \{ \vv \} \rangle$ by $\langle \vv \rangle $.

\subsection{$\rcd (K,N)$ spaces}  Throughout this text, we assume the reader is familiar with the basic theory of $\rcd (K,N)$ spaces, including the definition of Sobolev space $\sobo$ and  Laplace operator $\Delta$ on (infinitesimally Hilbertian) metric measure spaces. We refer the reader to \cite{lott-villani, sturm-i, sturm-ii, ambrosio-gigli-savare, ambrosio-gigli-mondino-rajala} for the relevant definitions and to \cite{erbar-kuwada-sturm} for an overview of different equivalent definitions of $\rcd (K,N)$ spaces.

We note that a large number of papers in the literature work with a condition known as $\rcd ^{\ast} (K,N)$, originally  introduced in \cite{bacher-sturm}.  It is now known, thanks to the work of Cavalletti--Milman \cite{cavalletti-milman} and Li \cite{li}, that this condition is equivalent to the $\rcd (K,N)$ condition.

While the topology of an $\rcd (K,N)$ space $(X,\mathsf{d}, \mm )$ can be quite complicated (see for example \cite{hupp-naber-wang, zhou}), Wang  showed that it is semi-locally-simply-connected \cite{wang}. Moreover, by the locality of the $\rcd (K,N)$ condition \cite[Section 3]{erbar-kuwada-sturm}, lifting $\mm$ to the universal cover $\tilde{X}$ yields a measure $\tilde{\mm}$ that makes $\tilde{X}$ an $\rcd (K,N)$ space  (see \cite{mondino-wei}).

\subsection{Equivariant Gromov--Hausdorff convergence}
 Throughout this paper, we assume the reader is familiar with the notions of Gromov--Hausdorff and pointed Gromov--Hausdorff convergence, as presented for example in \cite{burago-burago-ivanov}. 

We refer to \cite{gromov-polynomial} for the definition of convergence of maps and to \cite{fukaya-yamaguchi} for the definition of equivariant Gromov--Hausdorff convergence. A feature of this frameword is that once one has pointed Gromov--Hausdorff convergence of proper spaces, compactness of groups of isometries comes for free \cite[Proposition 3.6]{fukaya-yamaguchi}.

\begin{theorem}[Fukaya--Yamaguchi]\label{thm:fukaya-yamaguchi}
Let $(X_i, p_i)$ be a sequence of pointed proper metric spaces that converges in the pointed Gromov--Hausdorff sense to a pointed proper space $(X,p)$, and $\Gamma_i\leq \iso (X_i)$ a sequence of groups of isometries. Then, after taking a subsequence, the triples $(X_i, \Gamma_i, p_i)$ converge to a triple $(X, \Gamma , p )$ with $\Gamma \leq \iso (X)$. 
\end{theorem}

Another feature of equivariant Gromov--Hausdorff convergence is its compatibility with quotients \cite[Theorem 2.1]{fukaya-convergence}.

\begin{theorem}[Fukaya]\label{thm:fukaya}
Let $(X_i, p_i)$ be a sequence of pointed proper metric spaces and $\Gamma_i\leq \iso (X_i)$ a sequence of groups of isometries. If the sequence of triples  $(X_i, \Gamma_i, p_i)$ converges to a triple $(X, \Gamma , p )$ in the equivariant Gromov--Hausdorff sense, then the seuquences of  pointed quotient spaces $(X_i / \Gamma_i , [p_i])$ converges in the pointed Gromov--Hausdorff sense to $(X/\Gamma, [p])$.
\end{theorem}

\subsection{Measured Gromov--Hausdorff convergence}\label{sec:mgh} 
One of the main features of the class of  $\rcd (K,N)$ spaces is the measured Gromov--Hausdorff compactness.    
 Throughout this paper, we use the extrinsic approach, introduced in \cite{gigli-mondino-savare} inspired on the original approach to pointed Gromov--Hausdorff convergence \cite{gromov-polynomial}, and widely adopted in the literature (see for example \cite{mondino-naber, ambrosio-honda}). We say a triple $(X,\mathsf{d}, \mm )$ is a \emph{metric measure space} if $(X, \mathsf{d})$ is a complete separable geodesic space and $\mm$ is a boundedly finite (finite on bounded sets) measure on $X$.

\begin{definition}
In a metric space $Z$, we say a sequence of boundedly finite measures $\mm_i$ \emph{converges weakly} to a measure $\mm$ if for any bounded continuous function $f: Z \to \mathbb{R}$ with bounded support, one has
\[      \int_{Z}   f \,  \text{d} \mm _i \to \int_Z   f \, \text{d} \mm    .                 \]
\end{definition}

\begin{definition}\label{def:mgh}
We say a sequence of compact metric measure spaces  $(X_i, \mathsf{d}_i, \mm _i  )$  converges in the \emph{measured Gromov--Hausdorff} sense to a compact metric measure space $(X,\mathsf{d}, \mm  )$ if there is a separable metric space $(Z , \dd _Z )$ and isometric embeddings $\iota _i :  X_i  \to Z $, $\iota : X \to  Z $ such that for every $\varepsilon > 0 $ there is $i_0 (\varepsilon ) \in \mathbb{N}$ such that for every $i \geq i_0$,  $x \in X_i $, $y \in X $, there are $y' \in X_i $ and $x' \in X $ with 
\[   \dd_Z (\iota _i (x) , \iota (x')  ) , \dd_Z(\iota_i (y' ) , \iota (y)) < \varepsilon    ,                        \]
and the sequence of measures $ (\iota _i )_{\ast} ( \mm _i )  $ converges weakly to the measure $ ( \iota ) _{\ast} (\mm )$. 
\end{definition}

\begin{definition} 
We say a sequence of pointed metric measure spaces  $(X_i, \mathsf{d}_i, \mm _i , p_i )$  converges in the \emph{pointed measured Gromov--Hausdorff} sense to a pointed metric measure space $(X,\mathsf{d}, \mm , p )$ if there is a separable metric space $(Z , \dd _Z )$ and isometric embeddings $\iota _i :  X_i  \to Z $, $\iota : X \to  Z $ such that for every $\varepsilon > 0 $ and $R > 0 $ there is $i_0 (\varepsilon , R ) \in \mathbb{N}$ such that for every $i \geq i_0$,  $x \in B_R (p_i)$, $y \in B_R (p)$, there are $y' \in B_{R } (p_i)$ and $x' \in B_{R} (p)$ with 
\[   \dd_Z (\iota _i (x) , \iota (x')  ) , \dd_Z(\iota_i (y' ) , \iota (y)) < \varepsilon    ,                        \]
and the sequence of measures $ (\iota _i )_{\ast} ( \mm _i )  $ converges weakly to the measure $ ( \iota ) _{\ast} (\mm )$. 
\end{definition}

\begin{theorem}\label{thm:compactness}
 For each $i \in \mathbb{N}$, let  $(X_i,\mathsf{d}_i, \mm _i,p_i)$ be an  $\rcd (K - \varepsilon_i ,N)$ space with $\varepsilon _i \to 0$. Then after normalizing the measures $\mm_i$, one can find a subsequence that converges in the pointed measured Gromov--Hausdorff sense to an $\rcd (K,N)$ space $(X,\mathsf{d}, \mm , p)$.
\end{theorem}

\begin{proof}
The fact that one can find a convergent subsequence follows from Gromov's pre-compactness criterion  \cite[Proposition 5.2]{gromov-ms} combined with Prokhorov's compactness theorem (see for example \cite[Theorem 27.32]{villani}). The fact that the limit space is an  $\rcd (K,N)$ space was proven in \cite{gigli-mondino-savare} building upon \cite{lott-villani, sturm-i, sturm-ii, ambrosio-gigli-savare}. 
\end{proof}

\subsection{Dimension and volume convergence} 

Let $(X,\mathsf{d} ,\mm)$ be an $\rcd (K,N)$ space. In \cite{brue-semola}, Bru\'e--Semola showed that there is a unique integer $n \leq N$ with the property that there is a set $\mathcal{R}_n \subset X$ of full measure such that for all $x\in \mathcal{R}_n$, the pointed Gromov--Hausdorff limit $\lim_{\lambda \to \infty }(\lambda X, x)$ exists and is isometric to $(\mathbb{R}^n, 0)$. The integer $n$ is called  the  \emph{rectifiable dimension} of $X$.

\begin{theorem} \label{thm:dim}
    Let $(X,\mathsf{d} ,\mm)$ be an $\rcd (K,N)$ space. Then the following are equivalent.
    \begin{enumerate}[(i)]
        \item $X$ has rectifiable dimension $N$. \label{thm:dim-1}
        \item  $X$ has topological dimension $N$. \label{thm:dim-3}
        \item $N \in \mathbb{N}$ and $X$ has Hausdorff dimension $N$. \label{thm:dim-2}
        \item $N \in \mathbb{N}$ and $\mm = c \mathcal{H}^N$ for some $c>0$.\label{thm:dim-4}
    \end{enumerate}
\end{theorem}

\begin{proof}
Since reference measures on $\rcd (K,N)$ spaces are non-zero and locally finite, \eqref{thm:dim-4} implies \eqref{thm:dim-2}.  By the work of De Philippis--Gigli \cite[Theorem 1.12]{de-philippis-gigli},  \eqref{thm:dim-4} implies  \eqref{thm:dim-1}. The fact that either  \eqref{thm:dim-1} or  \eqref{thm:dim-2} implies \eqref{thm:dim-4} was proven in \cite[Theorem 2.20]{brena-gigli-honda-zhu}. Since the topological dimension is always at most the Hausdorff dimension \cite[Theorem VII 2]{hurewicz-wallman}, and the Hausdorff dimension of an $\rcd (K,N) $ space is at most $N$, then \eqref{thm:dim-3} implies \eqref{thm:dim-2}. The fact that \eqref{thm:dim-4} implies \eqref{thm:dim-3} follows from the work of Kapovitch--Mondino \cite[Theorem 1.7]{kapovitch-mondino}.
\end{proof}

 An important  feature of  $\rcd (K,N)$ spaces of maximal dimension is the continuity of volume, originally proven by Colding in the smooth setting \cite{colding}.

\begin{theorem}\label{thm:volume-continuity}
    Let $(X_i, \mathsf{d}_i, \mm_i, p_i )$ be a sequence of pointed  $\rcd (K, N)$ spaces that converges in the pointed measured Gromov--Hausdorff sense to a pointed $\rcd (K,N)$ space $(X, \mathsf{d}, \mm ,  p)$ of rectifiable dimension $N$. Then
    \begin{enumerate}[(1)]
        \item  For $i$ large enough, $\mm_i$ is a constant multiple of  $\mathcal{H}^N$. \label{eq:non-collapsed}
        \item    $\lim_{i \to \infty }  \mathcal{H}^N( B_1^{X_i} (p_i) ) = \mathcal{H}^N ( B_1 ^X(p ) ).     $\label{eq:volume-continuity}
    \end{enumerate}
\end{theorem}

\begin{proof}
    The first claim follows from the lower semi-continuity of the rectifiable dimension \cite[Theorem 1.5]{kitabeppu} combined with Theorem \ref{thm:dim}. The second claim follows from the work of De Philippis--Gigli on non-collapsed $\rcd (K,N)$ spaces \cite[Theorem 1.3]{de-philippis-gigli}.
\end{proof}

A consequence of Theorem \ref{thm:volume-continuity}   is that non-collapsing sequences of $\rcd (K,N)$ spaces admit no small groups of isometries.

\begin{definition}
    Let $(X_i, p_i)$ be a sequence of proper geodesic spaces. We say a sequence of groups $W_i \leq \iso (X_i) $ consists of \emph{small subgroups} if for each $R > 0 $ we have
    \[    \lim _{i \to \infty } \, \sup _{g \in W_i}\,  \sup _{x \in B_R(p_i)} d(gx,x) = 0 .         \]
    Equivalently, the groups $W_i$ are small if they converge to the trivial group in the equivariant Gromov--Hausdorff sense and 
    \[     \lim_{i \to \infty} \sup _{g \in W_i} d(gp_i, p_i) < \infty .          \] 
\end{definition}

\begin{corollary}\label{cor:no-small-subgroups}
    Let $(X_i, \mathsf{d}_i, \mm_i, p_i )$ be a sequence of pointed  $\rcd (K, N)$ spaces that converges in the pointed measured Gromov--Hausdorff sense to a pointed $\rcd (K,N)$ space $(X , \mathsf{d}, \mm  , p)$ of rectifiable dimension $N$. Then any sequence of small subgroups $W_i \leq \iso (X_i)$ is eventually trivial.
\end{corollary}

\begin{proof}
    Working by contradiction, after passing to a subsequence, we assume $W_i$ is non-trivial for each $i$. For $i$ large enough, the closure $\overline{W_i}$ is a non-trivial compact Lie group, so it contains a non-trivial finite group $F_i \leq \overline{W_i}$. By Theorem \ref{thm:volume-continuity}\eqref{eq:non-collapsed}, the groups $\iso (X_i )$ are measure preserving, so by the work of Galaz--Kell--Mondino--Sosa \cite[Theorem 1.1]{galaz-kell-mondino-sosa} the pushforward measure $\overline{\mm}_i$ makes $X_i / F_i $ an $\rcd (K,N)$ space.

    Since the groups $F_i$ are small, the sequence $(X_i/F_i , d_{X_i/F_i} , \overline{\mm}_i , [p_i] )$ also converges in the pointed measured Gromov--Hausdorff sense to $(X, \mathsf{d}, \mm , p)$. However, applying Theorem \ref{thm:volume-continuity}\eqref{eq:volume-continuity} twice we get
\[  \mathcal{H}^N (B_1^X(p)) =    \lim_{i \to \infty } \mathcal{H}^N (B_1^{X_i / F_i} ([p_i]))  =  \lim_{i \to \infty } \frac{1}{\vert F_i \vert  }  \mathcal{H}^N (B_1^{X_i} (p_i)) \leq \frac{1}{2}  \mathcal{H}^N (B_1^X(p)) ,     \]
which is a contradiction.
\end{proof}

\color{black}

\subsection{Topological stability} The following is a special case of  \cite[Theorem 2.9]{kapovitch-mondino}.

\begin{theorem}[Kapovitch--Mondino]\label{thm:reifenberg-weak}
{
Let $(X_i, \mathsf{d}_i, \mm_i,  p_i)$ be a sequence of pointed $\rcd ( - \frac{1}{i} ,N)$ spaces such that the sequence $(X_i, p_i)$ converges in the pointed Gromov--Hausdorff sense to $(\mathbb{R}^N, 0)$. Then there is a sequence of pointed Gromov--Hausdorff approximations $f_i : (X_i , p_i) \to ( \mathbb{R}^N , 0 ) $ such that for all $R > \delta > 0$, the restriction of $f_i$ to $B_R(p_i)$ is a homeomorphism with its image for $i$ large enough depending on $R,\delta$, and
        \[            B_{R - \delta }(0) \subset f_i (B_R(p_i)) .  \] 
}

\end{theorem}

\begin{corollary}\label{cor:manifold-reifenberg}
     Let $(X_i, \mathsf{d}_i, \mm_i,  p_i)$ be a sequence of pointed $\rcd ( - \frac{1}{i} ,N)$ spaces such that the sequence $(X_i, p_i)$ converges in the pointed Gromov--Hausdorff sense to $(\mathbb{R}^N, 0)$. If there are groups $\Gamma _ i \leq \iso (X_i) $ with 
     \[ \limsup_{i \to \infty } \, \diam (X_i / \Gamma_i ) < \infty, \]
     then $X_i$ is a topological manifold of dimension $N$ for $i$ large enough.
\end{corollary}

\begin{proof}
  Let $R  > \limsup_i \diam (X_i / \Gamma _ i )$. By Theorem \ref{thm:reifenberg-weak}, for $i$ large enough each point in $B_{R}(p_i)$  has a neighborhood homeomorphic to $\mathbb{R}^N$. Since also for $i$ large enough and each $x \in X_i$ there is an isometry $g \in \Gamma_i$ with $gx \in B_{R}(p_i)$ the result follows. 
\end{proof}

\begin{corollary}\label{cor:locally-contractible-reifenberg}
     Let $(X_i, \mathsf{d}_i, \mm_i,  p_i)$ be a sequence of pointed $\rcd ( - \frac{1}{i} ,N)$ spaces such that the sequence $(X_i, p_i)$ converges in the pointed Gromov--Hausdorff sense to $(\mathbb{R}^N, 0)$. Then for any $\varepsilon$, the following holds for $i$ large enough:
    \begin{itemize}
        \item For all $r \in ( 0 , 1]$, $x \in B_{1}(p_i)$, the ball $B_r(x)$ is contractible in $B_{r + \varepsilon }(x)$. 
    \end{itemize}
\end{corollary}

\begin{proof}
    Let $f_i : X_i \to X$ be the Gromov--Hausdorff approximations given by Theorem \ref{thm:reifenberg-weak}. Assuming the result fails, after passing to a subsequence, there are  sequences $x_i \in B_1(p_i)$, $r_i \in ( 0 , 1 ]$ for which the ball $B_{r_i} (x_i)$ is not contractible in $B_{r_i + \varepsilon } (x_i)$. After passing further to a subsequence, we can assume $r_i \to r$ for some $r \in [ 0 , 1]$ and $f_i (x_i) \to x$ for some $x \in B_1(0)$. 
    
    Then for $i$ large enough, $B_{r + \varepsilon /4}(x_i)$ is not contractible in $B_{r + 3 \varepsilon /4} (x_i)$. However, this is false, as for $i$ large enough one has
    \[    f_i (B_{r+ \varepsilon /4 }( x_i)) \subset  B_{r+ \varepsilon / 2} (x) \subset  f_i (B_{r+ 3 \varepsilon /4 }( x_i))  ,   \]
    so we can pull back via $f_i$ a nullhomotopy of $B_{r + \varepsilon / 2} (x)$.
\end{proof}

\subsection{Analysis and Spectral Estimates}
Let $(X,\mathsf{d},  \mm ) $ be a compact $\rcd (K,N)$ space of diameter $\leq D$. The spectrum  of its Laplacian $\Delta$ is discrete and non-positive, and one has an orthogonal direct sum decomposition 
\[      \sobo (X) = \bigoplus_{\lambda \geq 0} \{ f \in \sobo (X) \vert \Delta f = - \lambda f  \}   .  \]
For an eigenfunction $f \in \sobo (X) $ with $\Delta f = - \lambda f $, one has
\[     \Vert f \Vert _{\sobo} ^2 = (1 + \lambda ) \Vert f \Vert _{L^2} ^2 .  \]
Hence by \cite[Proposition 7.1]{ambrosio-honda-portegies-tewodrose} and the discussion preceding it, there is $C(K,N, D) > 0 $ for which
 \begin{equation}\label{eq:eigenest}
   \hspace{0.6cm} \| f \|_{L^\infty} \leq C \lambda ^{(N-2)/4} \Vert f \Vert _{\sobo }  ,\hspace{1.2cm} \| \nabla  f  \|_{L^\infty} \leq C  \lambda ^{ N/4} \Vert f \Vert _{\sobo }.
 \end{equation}
 
The Laplacian spectrum is continuous with respect to measured Gromov--Hausdorff convergence. This continuity was originally conjectured for Riemannian manifolds by Fukaya \cite{fukaya}, proven in that context and more general ones by Cheeger--Colding \cite{cheeger-colding-iii}, and extended to $\rcd (K,N)$ spaces by Gigli--Mondino--Savare \cite{gigli-mondino-savare} (see also \cite{ambrosio-honda} for an extension in non-compact spaces).

\begin{definition}\label{def:weakly}
Let $(X_i, \mathsf{d}_i, \mm_i ) $ be a sequence of $\rcd (K,N)$ spaces converging in the measured Gromov--Hausdorff sense to an $\rcd (K,N)$ space $(X,\mathsf{d}, \mm )$, and  let $(Z, \dd_Z)$, $\iota _i : X_i \to Z$, $\iota : X \to Z $ be as in Definition \ref{def:mgh}.  Given a sequence $ f  _i \in \sobo (X_i )$, we say it \emph{converges in} $\sobo$  to $ f \in \sobo  (X)$ if the sequence of measures $ ( \iota _i )_{\ast} ( f_i \mm _i ) $ converges weakly to $(\iota)_{\ast} (f \mm )$ and 
\[ \lim_{i \to \infty } \Vert f _i \Vert _{\sobo } = \Vert f \Vert _{\sobo } . \]
We say $f_i$ \emph{converges uniformly} to $f$ if for any choice of points $x_i \in X_i$, $x \in X$ with $\iota _i (x_i)  \to  \iota (x) $, one has $f _i (x_i) \to f (x)$. 
\end{definition}
\color{black}

\begin{theorem}[Gigli--Mondino--Savar\'e]\label{thm:spectrum-continuity}
    Let $(X_i, \mathsf{d}_i, \mm _i) $ be a sequence of compact $\rcd (K,N)$ spaces that converges in the measured Gromov--Hausdorff sense to a compact $\rcd (K,N)$ space $(X,\mathsf{d},\mm )$. Then
    \begin{enumerate}[(1)]
        \item  For each sequence of eigenfunctions $f _i \in \sobo (X_i)$ with 
        \[  \Delta f_i = - \lambda _i f_i  ,\hspace{1.3cm}  \limsup_{i \to \infty } \lambda_i < \infty  , \hspace{1cm} \limsup_{i \to \infty } \Vert f _i\Vert _{\sobo} < \infty , \]
        there is an eigenfunction $f \in \sobo (X)$ with $\Delta f = - \lambda f $ such that after passing to a subsequence, $\lambda _i \to \lambda$ and $f_i$ converges to $f$ both uniformly and in  $\sobo$.\label{thm:spectrum-continuity-1}

        \item  For each eigenfunction $f \in \sobo (X)$ with $\Delta f = - \lambda f $, there is a sequence of eigenfunctions $f_i \in \sobo (X_i)$ with $\Delta f_i = - \lambda _i f_i$ such that  $\lambda_i \to \lambda$ and $f_i$ converges to $f $ both uniformly and in $\sobo $. \label{thm:spectrum-continuity-2}
    \end{enumerate}
\end{theorem}

\begin{corollary}\label{cor:spectrum-continuity}
 Let $(X_i, \mathsf{d}_i, \mm _i) $ be a sequence of compact $\rcd (K,N)$ spaces that converges in the measured Gromov--Hausdorff sense to a compact $\rcd (K,N)$ space $(X,\mathsf{d},\mm )$. Then the sequence of Laplace spectra of the spaces $X_i$ counted with multiplicities as measures on $\mathbb{R}$ converge weakly  to the Laplace spectrum of $X$.    
\end{corollary}

\subsection{Canonical homeomorphisms} Eigenfunctions provide means of constructing bi-H\"older homeomorphisms between suitable non-collapsed $\rcd (K,N)$ spaces.  Theorem \ref{thm:canonical-homeomorphism} below was obtained in  \cite[Theorem 1.2.2]{honda-peng}, inspired by the work of Wang--Zhao \cite{wang-zhao}, and built upon the tools developed by Cheeger--Jiang--Naber \cite{cheeger-jiang-naber}.

\begin{definition}
    Let $(X_i, \mathsf{d}_i, \mm_i)$ be a sequence of compact $\rcd (K,N)$ spaces. A sequence of maps $\Phi_i  : X_i \to \mathbb{R}^k$ is said to be \emph{equi-regular} if for any smooth function $f \in C^{\infty} ( \mathbb{R}^k ) $, the composition $f \circ \Phi_i : X_i \to \R $ is in the domain of the Laplacian  for all $i$ and 
    \[     \sup_{i } \Vert  \Delta (f \circ \Phi_i ) \Vert _{\infty} < \infty  .  \]
\end{definition}

\begin{theorem}[Honda--Peng]\label{thm:canonical-homeomorphism}
Let $(X_i, \mathsf{d}_i, \mm _i) $ be a sequence of compact $\rcd (K,N)$ spaces that converges in the measured Gromov--Hausdorff sense to a compact $\rcd (K,N)$ space $(X,\mathsf{d},\mm )$.  Assume $(X,\mathsf{d})$ is isometric to a smooth $N$-dimensional Riemannian manifold, and consider a smooth isometric embedding $\Phi : X \to \mathbb{R}^k$, together with an open neighborhood $U \subset \mathbb{R}^k$ and a $C^{1,1}$-Riemannian submersion $\pi : U \to \Phi (X)$ with $\pi \vert _{\Phi (X)}= \Id _{\Phi (X)}$.   Then if $\Phi _i = (\varphi _{1,i}, \ldots , \varphi _{k, i} ) : X_i \to \mathbb{R}^k$ is a sequence of equi-regular maps converging uniformly to $\Phi$ such that the family of functions  $\Delta \varphi_{j,i} $ is equi-Lipschitz, then the sequence of maps 
\[ F_i : = (\Phi^{-1} \circ \pi \circ \Phi_i ) : X_i \to X\] 
satisfies
\[  (1 - \varepsilon _i ) \, \mathsf{d}_i( x,y ) ^{1 + \varepsilon_i} \leq \mathsf{d}( F_i (x), F_i (y)  ) \leq (1 + \varepsilon_i) \, \mathsf{d}_i(x,y) \, \text{ for all } x,y \in X_i   \]
for some $\varepsilon _i \to 0$. Furthermore, if the spaces $X_i$ are smooth Riemannian manifolds, and $\Phi_i $ is smooth for each $i$, then $F_i$ is a diffeomorphism for $i$ large enough. 
\end{theorem}

\begin{remark}\label{rem:equi-regular}
    Let $(X_i,\mathsf{d}_i ,\mm_i)$ be a sequence of compact $\rcd (K,N)$ spaces of uniformly bounded diameter and $\varphi _{j,i} \in \sobo (X_i)$ a family of eigenfunctions with $ \Delta \varphi_{j,i} =  - \lambda _{j,i} \varphi _{j,i} $ for $j \in \{ 1, \ldots , k \}$ and assume
    \[   \sup_{i,j} \, \lambda _{j,i} < \infty ,\hspace{3cm} \sup_{i,j} \,  \Vert \varphi_{j,i}  \Vert _2 < \infty .  \]
    Consider the map $\Phi_i = (\varphi_{1,i}, \ldots , \varphi_{k,i} ) : X_i \to \mathbb{R}^k $. Since for any smooth function $f \in C^{\infty} (\mathbb{R}^k)$, one has
    \[    \Delta (f \circ \Phi_i ) = \sum_{j = 1}^k  \left[ \frac{\partial f}{\partial x _j} \circ \Phi_i \right]  \Delta \varphi_{j,i} + \sum_{j,\ell = 1 }^k  \left[ \frac{\partial ^2 f}{ \partial x_j \partial x_{\ell} } \circ \Phi_i \right] \langle \nabla \varphi_{j,i} , \nabla \varphi_{\ell, i} \rangle  ,    \]
    then by  \eqref{eq:eigenest}, the maps $\Phi_i = (\varphi_{1,i}, \ldots, \varphi _{k,i} ) : X _ i \to \mathbb{R}^k$ are equi-regular and equi-Lipschitz. 
\end{remark}

\subsection{Group theory}\label{subsec:rank} Recall that a  finitely generated group $\Lambda$ is said to be \emph{polycyclic} if there is a sequence of subgroups
\[ 1 = \Lambda _0 \trianglelefteq  \cdots \trianglelefteq \Lambda _m = \Lambda    \]
with $\Lambda _ j / \Lambda _ {j-1} $ cyclic for each $j$. In such a case, $\rank (\Lambda )$ is defined to be the number of $j$'s for which the quotient $\Lambda _ j / \Lambda _{j-1}$ is isomorphic to $\Z$. By Schreier Theorem \cite{fraleigh} this number is well defined. The following properties come directly from the definition.
\begin{proposition} \label{pro:rank-properties-1}
If $\Lambda ^{\prime } \leq \Lambda $ is a finite index subgroup then $\rank (\Lambda ^{\prime } ) = \rank (\Lambda )$.
\end{proposition}
\begin{proposition} \label{pro:rank-properties-2}
For a short exact sequence 
\[ 1 \to K \to \Lambda \to \Gamma \to 1  \]
with $\Lambda$ a polycyclic group, we have
\[      \rank (\Lambda ) = \rank (K) + \rank (\Gamma ).  \]
\end{proposition}
For an arbitrary finitely generated group $G$,  $\rank (G)$ is defined as the infimum of $ \rank (\Lambda )$ among finite index polycyclic subgroups $\Lambda \leq G$. By Proposition \ref{pro:rank-properties-1}  if $G$ is polycyclic there is no conflict between the distinct definitions of $\rank ( G )$. 
 
\begin{lemma}\label{lem:b1-compute}
    Let $G$ be a group, $\Gamma $ a  polycyclic group, and $ \theta :  G \to \Gamma  $ a surjective morphism with finite kernel. Then $ \rank (G) =  \rank (\Gamma )$.  
\end{lemma}

\begin{proof}
Let $\Lambda \leq G$ be the kernel of the adjoint action on $\Ker ( \theta ) \trianglelefteq G $. Then there is a short exact sequence
\[  1 \to \Ker (\theta ) \cap \Lambda \to \Lambda \to \theta (\Lambda) \to 1.   \]
$\Ker (\theta) \cap \Lambda $ is finite abelian, hence polycyclic. Since $\Ker (\theta) \cap \Lambda$ and $\theta (\Lambda )$ are polycyclic,  so is $\Lambda$. Then by Propositions \ref{pro:rank-properties-1} and \ref{pro:rank-properties-2};
    \[   \rank (G) = \rank (\Lambda ) = \rank (\theta (\Lambda) ) = \rank (\Gamma ) .                   \]
\end{proof}

The ensuing lemma follows from classic simultaneous diagonalization.

\begin{lemma}\label{lem:simultaneous-diagonalization}
    Let $G$ be an abelian group acting by linear isometries on a finite dimensional real vector space $V$ of dimension $2N$. Assume $V = \oplus _{\ell = 1}^m E_{\ell} $ for some $G$-invariant subspaces $E_{\ell}$. Then  $V$ admits a (not necessarily unique!) decomposition  $   V = \oplus _{j = 1 }^N W_j  $ with each $W_j$ 2-dimensional, $G$-invariant, and satisfying
    \[   W_j = \bigoplus _{\ell = 1 }^m \, ( E_{\ell } \cap W_j ) .  \]
\end{lemma}

\begin{definition}
For a group $G$ and two subgroups $H_1, H_2 \leq G$, we denote by $[H_1, H_2] \leq G$ the subgroup generated by elements of the form $h_1h_2 h_1^{-1} h_2 ^{-1}$ with $h_1 \in H_1$, $h_2 \in H_2$. We set $G^{(0)} : = G$,  and define inductively $G^{(j + 1 )} : = [G^{(j)} , G]$. $G$ is said to be \emph{nilpotent} if $G^{(s)}$ is the trivial group for some $s \in \mathbb{N}$.
\end{definition}

\subsection{Limits of almost homogeneous spaces}
In a portion of this paper, we follow the construction performed in \cite{zamora-limits}, which is an adaptation to metric spaces of a construction by Breuillard--Green--Tao in \cite{breuillard-green-tao}. This construction begins with the following result, which establishes that discrete groups that act almost transitively on proper geodesic spaces are virtually nilpotent.

\begin{lemma}\label{lem:almost-homogeoeus-subgroups}
Let  $(X_i, p_i)$ be a sequence of pointed proper geodesic spaces that converges to a pointed proper semi-locally-simply-connected geodesic space $(X,p)$ in the pointed Gromov--Hausdorff sense, and $G_i \leq \iso (X_i)$ a sequence of discrete groups with $\diam (X_i / G_i ) \to 0$. Then there exists $s \in \mathbb{N}$ and a sequence of finite-index subgroups $ G_i  ^{\prime} \leq G_i $ with 
\begin{gather}
 \lim_{i \to \infty }  \, \, \,  \sup _{x \in X_i} \, \sup _{g \in (G_i ^{\prime } ) ^{(s)}}  d(gx, x )   = 0 , \label{eq:gi-nilpotent} \\
   \limsup_{i \to \infty} \, [G_i : G_i ^{\prime}]  < \infty   .  \label{eq:bounded-index}
\end{gather}
 \end{lemma}
\begin{proof}
 By \cite[Theorem 4.1]{zamora-limits}, $X$ is a nilpotent Lie group equipped with a sub-Finsler metric. Sub-Finsler metrics are locally doubling \cite{nagel-stein-wainger}, so  \cite[Remark 3.4]{zamora-limits} applies and \cite[Lemma 3.2]{zamora-limits} provides the desired groups $G_i^{\prime}$.
\end{proof}

 It is well known that if a Riemannian manifold admits a  nilpotent group of isometries acting transitively, then such group acts freely \cite{wilson}. Conditions \eqref{eq:gi-nilpotent} and \eqref{eq:bounded-index} imply that the groups $G_i'$ are almost nilpotent and act almost transitively, so it is not surprising that they act almost freely. This is the content of the following lemma.

\begin{lemma}\label{lem:almost-translations}
    Let $(X_i, p_i)$, $(X,p)$, $G_i$, $G_i^{\prime}$, be as in Lemma \ref{lem:almost-homogeoeus-subgroups}. Then after passing to a subsequence the groups $G_i^{\prime}$ converge in the equivariant Gromov--Hausdorff sense to a connected nilpotent group $\Gamma \leq \iso (X)$ acting freely and transitively.  
\end{lemma}

\begin{proof} 
    By Theorem \ref{thm:fukaya-yamaguchi},  after passing to a subsequence the groups $G_i^{\prime}$ converge to a group $\Gamma \leq \iso (X)$. By the proof of \cite[Theorem 3.1]{zamora-limits},  there is a nilpotent subgroup $ \mathcal{O} \leq \Gamma $ acting freely and transitively  on $X$, so we can identify $\mathcal{O}$ with  $X$. By \cite[Theorem 4.1]{zamora-limits},  $\mathcal{O} $ is a Lie group, and by  \cite[Proposition 5.2]{zamora-limits}, $ \Gamma$ acts freely on $X$. Since the action of $\mathcal{O}$ is transitive, the only possibility is that $\Gamma = \mathcal{O}$.
\end{proof}

\color{black}

\begin{corollary}\label{cor:free}
    Let $(X_i, p_i)$, $(X,p)$, $G_i$, $G_i^{\prime}$, be as in Lemma \ref{lem:almost-homogeoeus-subgroups}. If the groups $G_i$ are such that any sequence of small subgroups is eventually trivial, then $G_i^{\prime}$ acts freely for $i$ large enough. 
\end{corollary}

\begin{proof}
    Assuming the contrary, after passing to a subsequence there are $x_i \in X_i$ with non-trivial stabilizers $K_i \leq G_i^{\prime}$. Notice that since the groups $G_i^{\prime}$ act almost transitively, we can assume $x_i \in B_1(p_i)$ for $i$ large enough.  After again passing to a subsequence, $x_i$ converges to a point $x \in X$ and the groups $K_i$  converge to a group $K \leq \Gamma $ contained in the stabilizer of $x$. By Lemma \ref{lem:almost-translations}, $K$ is trivial, so the groups $K_i$ are small, contradicting the hypothesis. 
 \end{proof}

Lemma \ref{lem:almost-translations} is particularly powerful when paired with Lemma \ref{lem:euclidean-nilpotent-groups}, which follows from the proof of \cite[Corollary 4.2.7]{thurston}.

\begin{lemma}\label{lem:euclidean-nilpotent-groups}
    Let $\Gamma \leq \iso (\mathbb{R}^n)$ be a connected nilpotent group of isometries acting transitively. Then $\Gamma$ is the group of translations in $\mathbb{R}^n$. 
\end{lemma}

Lemma \ref{lem:malcev-construction} below gives an explicit description of the groups $G_i^{\prime}$.

\begin{definition}
  Let $G$ be a group, $u_1, u_2, \ldots , u_r \in G$, and $N_1, N_2, \ldots , N_r $ $\in \mathbb{R}^+$. The set $P(u_1, \ldots , u_r ; N_1, \ldots , N_r) \subset G$ is defined to be the set of elements that can be expressed as words in the $u_i$'s and their inverses such that the number of appearances of $u_i$ and $u_i^{-1}$ is not more than $N_i$.  We then say that  $P(u_1, \ldots , u_r ; N_1, \ldots , N_r)$  is a \emph{nilprogression in $C$-normal form} for some $C>0$ if it also satisfies the following properties:
\begin{enumerate}[(1)]
\item  For all $1 \leq i \leq j \leq r$, and all choices of signs, we have
\begin{center}
$  [ u_i^{\pm 1} , u_j^{\pm 1} ]  \in P \left(  u_{j+1} , \ldots , u_r ; \dfrac{CN_{j+1}}{N_iN_j} , \ldots , \dfrac{CN_r}{N_iN_j}  \right). $
\end{center}\label{condition:n1}
\item The expressions $ u_1 ^{n_1} \ldots u_r^{n_r} $ represent distinct  elements as $n_1, \ldots , n_r$ range over the integers with  $\vert n_1 \vert \leq  N_1/C , \ldots , \vert n_r \vert \leq  N_r/C$.\label{condition:n2}
\item One has 
\[  \frac{1}{C} (2\lfloor N_1 \rfloor + 1 ) \cdots ( 2\lfloor N_r \rfloor + 1 ) \leq \vert P \vert  \leq C (2 \lfloor N_1 \rfloor + 1 ) \cdots ( 2\lfloor N_r \rfloor + 1 ) .               \]\label{condition:n3}
\end{enumerate} 
For a nilprogression $P$ in $C$-normal form, and $\varepsilon \in (0,1)$, the set  $P( u_1, \ldots , u_r ; $ $ \varepsilon N_1, \ldots , \varepsilon N_r   )$ also satisfies conditions \eqref{condition:n1} and \eqref{condition:n2}, and we denote it by $\varepsilon P$. We define the \emph{thickness} of $P$ as the minimum of $N_1, \ldots , N_r$ and we denote it by $ \thi (P)$. The set $\{ u_1 ^{n_1 }\ldots u_r^{n_r} \vert \vert n_i \vert \leq N_i/C  \}$ is called the \emph{grid part of }$P$, and is denoted by $G(P)$.
\end{definition}

\begin{lemma}\label{lem:malcev-construction}
Let $(X_i, p_i)$, $(X,p)$, $G_i$, $G_i^{\prime}$ be as in Lemma \ref{lem:almost-homogeoeus-subgroups} and $n$ the topological dimension of $X$. Then for $i$ large enough, there are torsion-free nilpotent groups $\tilde{\Gamma}_i$ generated by elements $\tilde{u}_{1,i}, \ldots , \tilde{u}_{n,i} \in \tilde{\Gamma}_i$, and $N_{1,i}, \ldots , N_{n,i} \in \mathbb{R}^+$ with the following properties:
\begin{enumerate}[(i)]
        \item The set map $ \alpha _i  : \mathbb{Z}^n \to \tilde{\Gamma } _i$ given by 
        \[  \alpha_i (m_1 , \ldots , m_n) = \tilde{u}_{1,i}^{m_1}\cdots \tilde {u}_{n,i}^{m_n}       \]
        is a bijection.  Moreover, there are polynomials $Q_i : \mathbb {R}^n \times \mathbb{R}^n\to \mathbb{R}^n$ of degree $\leq d(n)$ such that the product in $\tilde{\Gamma}_i$ is given by 
        \[g_1 \cdot g_2 = \alpha _i ( Q_i ( \alpha _i ^{-1} (g_1) , \alpha _i ^{-1} (g_2) )  ).  \]\label{conclusion:malcev-1}
        \item  There are small normal subgroups $W_i \trianglelefteq G_i^{\prime }$ and surjective group morphisms 
        \[\Phi _i : \tilde {\Gamma }_ i \to \Gamma_i : = G_i^{\prime } / W_i\]
        with $\Ker (\Phi_i)$ containing an isomorphic copy of $\pi_1(X)$.  \label{conclusion:malcev-2}
        \item There is $C > 0 $ such that if $u_{j,i} : = \Phi _i (\tilde{u}_{j,i}) $ for each $j \in \{ 1, \ldots , n\}$, the set
        \[   P_i : = P (  u_{1,i}, \ldots , u_{n,i} ; N_{1,i}, \ldots , N_{n,i} )  \subset \Gamma_i       \]
        is a nilprogression in $C$-normal form with $\thi (P_i ) \to \infty $.  \label{conclusion:malcev-3}
        \item  \label{conclusion:malcev-4} For each $\varepsilon > 0 $ there is $\delta > 0 $ such that 
        \begin{align*}
            G(\delta P_i )     \subset \{ g \in \Gamma_i \, \vert  \,  &d(g[p_i], [p_i] ) \leq \varepsilon \}  , \\
              \{ g \in \Gamma_i \, \vert  \, d(g [p_i ], [p_i&] ) \leq \delta \}  \subset     G( \varepsilon  P_i   ) 
        \end{align*}
        for $i$ large enough, where we are considering the action of $\Gamma _ i$ on $X_i / W_i $,  
\end{enumerate}
\end{lemma}

\begin{proof}
    In \cite[Section 6]{zamora-limits} first the small subgroups $W_i\trianglelefteq G_i^{\prime }$ are constructed (using the escape norm from \cite{breuillard-green-tao}) in such a way that the quotients $G_i^{\prime }/ W_i$ have no small subgroups. Then in \cite[Section 9]{breuillard-green-tao} the nilprogressions $P_i \subset \Gamma _ i$ are obtained via a short basis construction. \eqref{conclusion:malcev-3} and the first item of \eqref{conclusion:malcev-4} follow from \cite[Theorem 9.3]{breuillard-green-tao} and its proof, while the second item of  \eqref{conclusion:malcev-4} is proven in \cite[Theorem 7.1]{zamora-limits}.  In \cite[Section 8]{zamora-limits}, \eqref{conclusion:malcev-1} is obtained by an application of Malcev Theorem, and \eqref{conclusion:malcev-2} follows from \cite[Proposition 9.1]{zamora-limits}. Note that although the results in \cite{breuillard-green-tao, zamora-limits} are stated in terms of ultralimits and non-standard analysis, they are equivalent to the above formulations by \L os's Theorem (see for example \cite[Appendix C]{breuillard-green-tao}). 
\end{proof}

\begin{remark}\label{rem:rank-equality}
    Combining Lemma \ref{lem:malcev-construction} with Lemma \ref{lem:b1-compute} and Proposition  \ref{pro:rank-properties-2}  one obtains
    \begin{equation}\label{eq:rank-string}
   \begin{split}
        \rank (G_i^{\prime} ) & =  \rank (\Gamma _i )  \\
        & = \rank (\tilde {\Gamma}_i) - \rank (\Ker (\Phi _ i)) \\
        & \leq  n.
   \end{split}
    \end{equation}
If equality holds throughout  \eqref{eq:rank-string}, then $\rank (\Ker (\Phi_i) ) = 0$, and since the groups $\tilde{\Gamma}_i$ are torsion-free, that would imply $\Ker (\Phi _ i ) = 0$, so $X$ is simply connected and $\tilde {\Gamma }_i = \Gamma_i$.
\end{remark}

\begin{remark}\label{rem:nss-case}
    If the groups $G_i^{\prime}$ admit no small subgroups, then $\Gamma_i = G_i ^{\prime}$ for $i$ large enough.
\end{remark}

\section{Topological Lemmas}\label{sec:topology}

In general, when two pointed proper metric spaces $(X,p)$, $(Y,q)$ are pointed Gromov--Hausdorff close, there is no global approximation $X \to Y$. However, when a group $G$ acts co-compactly on both spaces and the triples $(X,p,G)$, $(Y,q,G)$ are equivariantly Gromov--Hausdorff close, it is possible to construct a global almost $G$-equivariant approximation $X \to Y$ (see for example \cite[Theorem 4.1]{wang-limit}). 

\begin{lemma}\label{lem:local-to-global-approximation}
For each $\varepsilon > 0 $, there is $\delta > 0 $ such that the following holds. Let $X$, $Y$ be proper geodesic spaces, $G_{X} \leq \iso (X)$, $G_Y \leq \iso (Y) $ discrete groups, $p \in X$, $q \in Y$, a group isomorphism $\varphi : G_X \to G_Y$, and a map $f : B_{10}(p) \to B_{10}(q)$ such that $f(p) = q$ and
\begin{enumerate}[(1)]
    \item $\diam ( X / G_X ),$ $\diam (Y / G_Y )  < 1$. \label{condition:f-1}
    \item  For all $x_1, x_2 \in B_{10}(p)$,  
    \[   \vert d(f(x_1), f(x_2)) - d(x_1, x_2) \vert \leq \delta .                 \]\label{condition:f-2}
    \item  For all $y \in B_1(q)$, there is $x \in B_2 (p)$ with
    \[     d(f(x) ,y) \leq \delta .        
    \]\label{condition:f-2.5}
    \item \color{black} If $x \in B_{10}(p)$, $h \in G_Y$ are such that $hf(x) \in B_{5}(q)$, then 
    \[   \varphi^{-1}(h) x \in B_{10}(p).    \] \label{condition:f-3}
    \item  If $x \in B_{10}(p)$, $g \in G_X$ are such that $gx \in B_{10}(p)$, then 
    \[   d(f(gx), \varphi (g) f(x) ) \leq \delta .    \]\label{condition:f-4}
\end{enumerate}
 Then there is $\tilde {f} : X \to Y$ such that
 \begin{enumerate}[(i)]
     \item     For all $x \in X$, $g \in G_X$, 
     \[         d(  \varphi (g)  \tilde{f} (x)  , \tilde{f} (gx) ) \leq \varepsilon .                  \]\label{conclusion:f-tilde-eq}     
     \item For all $x_1, x_2 \in X$ with $\min \{ d(x_1, x_2) ,  d(\tilde{f} (x_1),\tilde{f} (x_2)) \}  \leq 2 $  one has 
     \[   \vert d ( \tilde {f} (x_1) \tilde{f} (x_2) ) - d(x_1, x_2) \vert \leq \varepsilon.     \] \label{conclusion:f-tilde-1}
     \item  For all $y \in Y$, there is $x \in X$ with $d ( \tilde{f} x , y ) \leq \varepsilon$.\label{conclusion:f-tilde-2}
 \end{enumerate}
 Moreover, if  
 \begin{equation}\label{condition:riemannian-sec-inj}
  Y\text{ is Riemannian,} \hspace{2cm} \vert \sec (Y) \vert \leq 1,   \hspace{2cm}  \inj (Y) \geq 1,   
 \end{equation}
   and $f $ is continuous, then $\tilde{f}$ can be taken to be continuous and equivariant; that is, for all $x \in X$, $g \in G_X$, 
 \begin{equation*}
 \varphi (g) \tilde{f} (x) = \tilde{f} (gx) .  
 \end{equation*}
\end{lemma}

\begin{proof}
    For each $g \in G_X$, set $\theta _g : X \to \mathbb{R}$ as 
    \[   \theta_g(x) : = \max \{ 0, 2 - d(x, gp) \}.           \]
    Then define $\Theta : X \to \mathbb{R}$ as 
    \[\Theta (x) = \sum_{g \in G_X} \theta_g (x), \]
    which by \eqref{condition:f-1} is always positive.  Also for each $g \in G_X$ define $\rho _g : X \to \mathbb{R}$ as 
    \[   \rho _g (x)  : = \theta _g (x) / \Theta (x),  \]
    and with this, the map $\mu : X \to \prob (Y)$ as
    \[    \mu _x : = \sum_{g \in G_X} \rho _g (x) [[  \varphi (g) ( f ( g^{-1} x ) ) ]],          \]
    where for each $y \in Y$, we denote by $[[y]] \in \prob (Y)$ the delta measure supported at $y$.  Notice that $\rho_g (x) \neq 0$ means that $g^{-1}x \in B_1(p)$, so $\mu _x$ is well defined even if $f(g^{-1}x)$ is not for most $g \in G_X$. It follows from the definitions that  
    \[  \rho _g (hx) = \rho_{h^{-1}g} (x)  \text{  for all }x \in X,\,  g,h \in G_X. \]
    Then 
    \begin{equation}\label{eq:mu-equivariant}
        \begin{split}
    \mu _{hx}  &  =   \sum_{g \in G_X} \rho _g (hx) [[  \varphi (g) ( f ( g^{-1}h x ) ) ]]  \\
    &   =   \sum_{g \in G_X} \rho _{h^{-1 }g} ( x) [[  \varphi (g) ( f ( (h^{-1 }g)^{-1} x ) ) ]] \\
        &   =  \sum _{u \in G_X}  \rho _{ u } (x) [[  \varphi (h u ) (f (u ^{-1} x)  ) ]] \\
        & = \varphi ( h) _{\ast} \sum _{u \in G_X}  \rho _{ u } (x) [[  \varphi ( u ) (f (u ^{-1} x)  ) ]] \\
        & = \varphi (h ) _{\ast} \mu _x ,             
        \end{split}
    \end{equation}
    so $\mu $ is an equivariant map. Also, if $x \in X$,  $g, h \in G_X$ are such that  $\rho _{g} (x), \rho _{h} (x) > 0 $, then $g^{-1}(x), h^{-1} (x) \in B_2(p)$, and by  \eqref{condition:f-4} we have
    \begin{align*}
         d( \varphi (g )  f(  &g^{-1} x ) ,  \varphi (h ) f(  h^{-1} x )  )    =   d (  f(  g^{-1} x ) ,  \varphi (g^{-1} h) f(  h^{-1} x ) ) \\
         & =  d (  f(  ( g^{-1} h  ) h ^{-1 } x ) ,  \varphi (g^{-1} h) f(  h^{-1} x ) ) \\
         & \leq \delta ,
    \end{align*}
    so 
    \begin{equation}\label{eq:diameter-supp-mu}
    \diam ( \supp (\mu _x )) \leq \delta \text{ for all }x \in X.    
    \end{equation}
    Then define $\tilde {f} : X \to Y$ as any map with
    \begin{equation}\label{eq:f-tilde-def}
    \tilde{f} (x) \in \supp (\mu _x ).  
    \end{equation}
    By \eqref{eq:diameter-supp-mu},  different choices for $\tilde{f}$ yield functions that are uniformly $\delta $-close, and \eqref{conclusion:f-tilde-eq} follows from  \eqref{eq:mu-equivariant}. Now if $x_1, x_2 \in X$ are such that  $ d( x_1, x_2 ) \leq 5,$ take  $g_1, g_2 \in G_X$ with $g^{-1}_1x_1 , g^{-1}_2x_2 \in B_1(p)$. Then $g_1^{-1} x_2 \in B_{10} (p)$ so by  \eqref{condition:f-2} and \eqref{condition:f-4} we have  
    \begin{align*}
     d( \varphi (g_1) f(&g_1^{-1} x_1) , \varphi (g_2) f(g_2^{-1}x_2)  )   =  d(   f ( g_1 ^{-1} x_1  ) , \varphi ( g_1 ^{-1}g_2  ) f(  g_2 ^{-1} x_2 )   )\\
    & =  d (  f ( g_1 ^{-1} x_1 ) , f (  g_1 ^{-1} x_2 ) ) \pm \delta \\
    & = d( g_1^{-1} x_1 , g_1^{-1}x_2 )  \pm 2 \delta \\
    & = d(x_1, x_2 ) \pm 2 \delta .
    \end{align*}
    Combining this with \eqref{eq:diameter-supp-mu} we obtain
    \begin{equation}\label{eq:f-tilde-gh}
        \vert d ( \tilde{f} (x_1) , \tilde{f} (x_2)  ) - d(x_1, x_2) \vert \leq 4 \delta .   
    \end{equation}   
    On the other hand, if $x_1, x_2$ are such that $d( \tilde{f} (x_1), \tilde{f}(x_2)) \leq 2$, then there are $g_1, g_2 \in G_X$ such that $\rho_{g_1}(x_1), \rho _{g_2 } (x_2) > 0 $, and 
    \[    d(  \varphi ( g_1 ) f(g_1^{-1}x_1)     , \varphi (g_2) f(g_2 ^{-1} x_2) ) \leq 3.                            \]
    Using \eqref{condition:f-2} and the fact that $f(p)  = q$, this implies 
    \[        d(  \varphi ( g_2^{-1}  g_1 )  f( g_1^{-1}x_1), q )  \leq 5, \]
    so by \eqref{condition:f-3}, we get $ g_2^{-1} g_1 g_1^{-1} x_1 = g_2^{-1}x_1 \in B_{10}(p)$. Hence by \eqref{condition:f-2}, 
    \begin{align*}
    d( x_1, x_2 ) & =  d ( g_2^{-1}  x_1 , g_2^{-1} x_2  ) \\
     & \leq d(  f(  g_2 ^{-1} x_1  ) , f ( g_2 ^{-1} x_2 )  )  +  \delta \\
     & \leq d(  \varphi ( g_2 ^{-1}  g_1 ) f(  g_1 ^{-1} x_1  ) , f ( g_2 ^{-1} x_2 )  )  + 2 \delta   \\
     & \leq 5,
    \end{align*}
    so  \eqref{eq:f-tilde-gh} applies, proving  \eqref{conclusion:f-tilde-1}.  For $y \in Y$, by  \eqref{condition:f-1} there is $g \in G_X$ with 
    \[ \varphi (g) ^{-1} y \in B_1(q) .\] 
    By \eqref{condition:f-2.5}, there is $x \in B_2 (p)$ with $d(f(x), \varphi (g)^{-1} y ) \leq \delta$.  Since $\rho_g(gx) > 0$, we have
    \[    d ( \tilde {f} (gx), y ) \leq d(  \varphi (g) f( g^{-1}g x   ) , y ) + \delta \leq 2 \delta , \]
  establishing  \eqref{conclusion:f-tilde-2}. 
    
    Now assume $Y$ satisfies \eqref{condition:riemannian-sec-inj} and $f$ is continuous. Then instead of defining $\tilde{f} $ by  \eqref{eq:f-tilde-def},  define $\tilde{f} (x) \in Y $ to be the center of mass of $\mu_x$. As the functions $\rho _g$ are all continuous, $\mu _ x$ depends continuously and equivariantly on $x$, so the result follows.
\end{proof}

When two geodesic spaces $X$ and $Y$ are Gromov--Hausdorff close, very little can be said about their topologies (see for example \cite{ferry-okun}). However, if $X$ satisfies nice local contractibility properties and admits a continuous Gromov--Hausdorff approximation $X \to Y$, some of its topological features are dominated by the corresponding ones of $Y$ (see for example \cite{petersen}). Lemmas \ref{lem:transfer-trivial-homology} and \ref{lem:transfer-simply-connected} below are in the spirit of this phenomenon.

\begin{lemma}\label{lem:transfer-trivial-homology}
 For each $m \in \mathbb{N}$, there is $\varepsilon > 0 $ such that the following holds. Let $X$, $Y$ be proper geodesic spaces and maps $f: X \to Y$, $h : Y \to X$ such that 
 \begin{enumerate}[(1)]
     \item For each $r \in ( 0 , 1 ]$, $x \in X$, the ball $B_r(x)$ is contractible in $B_{r + \varepsilon }(x)$.\label{condition:contractibility-of-balls}
     \item  For all $y_1, y_2 \in Y$ with $d(y_1, y_2) \leq 1$, one has $d( hy_1, hy_2 ) \leq d(y_1, y_2 ) + \varepsilon $.\label{condition:h-non-expanding}
     \item  For all $x \in X$, one has $d( h (f (x)) , x ) \leq \varepsilon$.\label{condition:hf-id}
     \item  $f$ is continuous.\label{condition:f-continuous}
 \end{enumerate}
 Then the map $f_{\ast} : H_m (X) \to H_m (Y)$ is injective. In particular, if $H_m(Y) $ is trivial, then so is $H_m (X)$. 
\end{lemma}

Ideally, we would like the proof of Lemma \ref{lem:transfer-trivial-homology} to go like this: $h \circ f \approx \Id _X $, thus the identity map $H_m (X) \to H_m (X)$ factors as 
\[      H_m (X) \xrightarrow{f_{\ast}} H_m(Y) \xrightarrow{h_{\ast}} H_m (X)  .  \]
This argument fails as $h$ is not necessarily continuous, so $h \circ f$ is not exactly  homotopic to the identity on $X$. However, condition \eqref{condition:contractibility-of-balls} allows us to get around these issues. We now turn to the actual proof (cf. \cite[Proposition 1.3]{yamaguchi}).

\begin{proof}
    Assume $\varepsilon < \frac{1}{4m + 3} $.  For each $x_0, \ldots , x_k \in X$ with $\diam ( \{ x_0, \ldots , x_k  \} ) \leq 3 \varepsilon$ and $k \leq m+1$, we will construct a singular $k$-simplex $[x_0, \ldots , x_k]$ in $X$. For each $x_0 \in X$, we denote by $[x_0]$ the singular $0$-simplex consisting of the point $x_0$. For each $x_0, x_1 \in X$ with  $ d(x_0,x_1) \leq 3 \varepsilon ,$  let  $[x_0, x_1]$ be a geodesic from $x_0$ to $x_1$ seen as a singular 1-simplex. 

    Fix $k \leq m $ and assume by induction that for each $j \in \{ 1, \ldots , k \}$, and any points $x_0, \ldots , x_j \in X$ with $\diam (\{ x_0, \ldots , x_j \}) \leq 3 \varepsilon $, there is a singular $j$-simplex $[x_0, \ldots , x_j] $ in $X$ contained in $B_{4j \varepsilon }(x_0)$ and compatible with the lower dimensional simplices; that is, 
    \begin{equation}\label{eq:compatibility-1}
        \partial [ x_0, \ldots , x_j ] = \sum _{\ell = 0 }^j \,  (-1) ^ {\ell} \, [x_0, \ldots , \hat{x}_{\ell} , \ldots , x_j]    .  
    \end{equation}        
    Take $ x_0 , \ldots , x_{k + 1} \in X $ with $\diam (\{ x_0, \ldots , x_{k+1} \} ) \leq 3 \varepsilon$. Then  the  already constructed simplices $[x_0, \ldots , \hat{x}_j , \ldots , x_{k+1} ]$ are contained in $B_{(4k + 3)\varepsilon}(x_0)$ and by  \eqref{eq:compatibility-1} they satisfy 
    \begin{equation*}
        \sum_{j = 0} ^{k+1} \, (-1) ^{j} \, \partial  \,  [x_0, \ldots , \hat{x}_j , \ldots  , x_{k+1} ] =  0      
    \end{equation*}        
    so by condition  \eqref{condition:contractibility-of-balls} there is a singular $( k + 1 )$-simplex $[x_0, \ldots , x_{k+1}] $ contained in $B_{4(k+1)\varepsilon}(x_0)$ that restricted to its facets gives the ones already constructed; that is, 
    \begin{equation}\label{eq:h-chain}
        \partial [x_0, \ldots , x_{k+1} ] =  \sum_{j=0}^{k+1} \,  (-1)^j \,   [ x_0, \ldots , \hat{x}_j, \ldots , x_{k+1} ]  .     
    \end{equation}            
    For each singular $k$-simplex $\tau : \Delta ^k \to Y $, we define $h_{\sharp} ( \tau ) $ to be the singular $k$-simplex in $X$ given by $[ h \tau (e_0), \ldots , h \tau (e_k)  ]$, where $e_0, \ldots , e_k$ are the vertices of $\Delta ^k$. By condition \eqref{condition:h-non-expanding} this can be done provided $\diam ( \tau (\Delta ^k ) ) \leq \varepsilon $ and $k \leq m+1$. This definition expands linearly to all chains of the form $\psi = \sum_{\tau \in T} n_{\tau} \tau $ where  $T$ is a finite set of singular $k$-simplices of diameter $\leq \varepsilon $ and  $n _{\tau } \in \mathbb{Z}$ for each $\tau$,  as 
    \[ h_{\sharp} (\psi) : = \sum_{\tau \in T} n_{\tau} h_{\sharp }(\tau ). \] 
    Notice that by  \eqref{eq:h-chain}, we also have $ \partial h_{\sharp} (\psi ) = h_{\sharp} (\partial \psi ) $.

     Consider an $m$-cycle  $\varphi = \sum_{\sigma \in S} n_{\sigma} \sigma$ in $X$ representing a homology class in $\Ker (f_{\ast})$, where $S$ is a finite set of singular $m$-simplices, and $n_{\sigma} \in \mathbb{Z} $ for each $\sigma \in S$. Our goal is to show that $\varphi$ is a boundary.  By \eqref{condition:f-continuous}, the cycle 
     \[ f_{\sharp} ( \varphi )  = \sum_{\sigma \in S} n_{\sigma } f _{\sharp } (\sigma ) : = \sum_{\sigma \in S} n_{\sigma } (f \circ \sigma ) \] 
     is well defined and by hypothesis is a boundary. Take $\psi$ a singular $(m+1)$-chain in $Y$ with $\partial \psi = f_{\sharp} ( \varphi ) $. By applying barycentric subdivision simultaneously on both sides, we can assume the diameter of any singular simplex featured in the chains $\varphi$ and $\psi $ has diameter $\leq \varepsilon$, so $h_{\sharp} ( \psi )$ is well defined.

    Then by condition \eqref{condition:hf-id}, the boundary 
    \[ \partial h_{\sharp }(\psi ) = h_{\sharp } ( \partial \psi ) = h_{\sharp } f_{\sharp } ( \varphi  )  = \sum_{\sigma \in S} n_{\sigma } h_{\sharp } f_{\sharp } ( \sigma ) \]
    is uniformly close to $\varphi$. That is, for each $\sigma \in S$, the maps $\sigma ,  \, h_{\sharp} f_{\sharp} (\sigma ) : \Delta ^ m \to X $ are at uniform distance $\leq 8 m  \varepsilon .$ Thus if $\varepsilon $ is small enough (depending on $m$) one can (inductively on the skeleta just like at the beginning of the proof)  construct a homotopy  between $ \varphi $ and $h_{\sharp } f_{\sharp } ( \varphi )$. This implies that $\varphi$ is also a boundary. 

\end{proof}

\begin{lemma}\label{lem:transfer-simply-connected}
 Let $X$, $Y$ be proper geodesic spaces and maps $f: X \to Y$, $h : Y \to X$ that satisfy conditions \eqref{condition:contractibility-of-balls}, \eqref{condition:h-non-expanding}, \eqref{condition:hf-id},  \eqref{condition:f-continuous} from Lemma \ref{lem:transfer-trivial-homology} with $\varepsilon = 1/100$. Then $f_{\ast} : \pi_1 (X) \to \pi_1(Y)$ is injective. In particular, if $Y$ is simply connected then so is $X$. 
\end{lemma}

\begin{proof}
    The proof is essentially the same as for Lemma \ref{lem:transfer-trivial-homology} (cf. \cite[Theorem 2.1]{sormani-wei}). Given a loop $\varphi : \mathbb{S}^1 \to X$ representing an element of $\Ker (f_{\ast})$, there is a disk $\psi : \mathbb{D} \to Y  $ with $\psi \vert _{ \mathbb{S}^1} = f \circ \varphi$. Triangulate $\mathbb{D}$ in such a way that the image of each face of the triangulation under $\psi$ has diameter $\leq 1/100$. Then for each face with vertices $a, b, c \in \mathbb{D}$, construct the singular 2-simplex $[ h(\psi (a)) , h(\psi (b)) , h(\psi (c) )  ] $ in $X$ as in the proof of Lemma \ref{lem:transfer-trivial-homology}. Putting these simplices together, we obtain a continuous map $h_{\sharp} \psi : \mathbb{D} \to X $ with  $h_{\sharp } \psi \vert _{\mathbb{S}^1 }$ uniformly close to $\varphi$. Applying condition \eqref{condition:contractibility-of-balls} once more we obtain that $\varphi$ is homotopic to $h_{\sharp} \psi \vert _{\mathbb{S}^1}$, hence nullhomotopic.
\end{proof}

\begin{remark}
    At first, condition \eqref{condition:f-continuous} in Lemma \ref{lem:transfer-trivial-homology}  may seem like an unnecessary technicality. However, simple examples such as Berger metrics on $\mathbb{S}^3$ converging to a round metric on $\mathbb{S}^2$ show that this condition is actually necessary. 
\end{remark}

\begin{remark}
    With essentially the same proof, condition \eqref{condition:contractibility-of-balls} in lemmas \ref{lem:transfer-trivial-homology} and \ref{lem:transfer-simply-connected} can be replaced by the existence of a contractibility function in the sense of \cite{petersen} (on which $\varepsilon$ depends). Since we will not need such version, we leave it as an exercise for the reader.
\end{remark}

\color{black}

The following lemma is a basic result of algebraic topology.

\begin{lemma}\label{lem:contractible}
Let $X$ be a simply connected topological manifold of dimension $n$. If $H_m (X) = 0$ for all $m \in \{ 1, \ldots, n\} $, then $X$ is contractible.
\end{lemma}
\begin{proof}
    By \cite[Proposition 3.29]{hatcher} all homology groups of $X$ vanish, so by Hurewicz Theorem \cite[Theorem 2.1]{hutchings} and induction all homotopy groups vanish as well. By the work of Milnor \cite[Corollary 1]{milnor-cw}, $X$ is homotopy equivalent to a CW-complex and Whitehead Theorem \cite[Theorem 4.5]{hatcher} applies, so  $X$ is contractible.
\end{proof}

As mentioned in Section \ref{subsec:outline}, in order to prove that a closed manifold is homeomorphic to an infranilmanifold, it is enough to verify that it is aspherical and it has the correct fundamental group.   
The following argument was extracted from \cite{kapovitch}.

\begin{theorem}\label{thm:borel-nilpotent}
    Let $X$ be a closed aspherical manifold with $\pi_1(X)$ virtually nilpotent. Then $X$ is homeomorphic to an infranilmanifold.
\end{theorem}

\begin{proof}
    By  the work of Lee--Raymond \cite{lee-raymond}, there is an infranilmanifold $M$ with $\pi_1(M) = \pi_1(X)$. Then from the proof of the geometrization conjecture by Perelman in dimension $3$, the work of Freedman--Quinn in dimension $4$ \cite{freedman-quinn}, and the work of Farrell--Hsiang in dimensions 5 and higher \cite{farrell-hsiang}, the Borel conjecture holds for such groups, thus the manifolds $X$ and $M$ are homeomorphic.  
\end{proof}

\section{Asphericity Theorem}\label{sec:aspherical}

As mentioned in Section \ref{subsec:outline}, a crucial step to prove Theorem \ref{thm:max-rank-nilmanifold} is to show that in case of equality, the universal cover is a contractible manifold. We now prove something slightly stronger.

\begin{theorem}\label{thm:contractible}
    For each $K \in \mathbb{R}$, $N \geq 1$, there is $\varepsilon > 0 $ such that if $(X, \mathsf{d} , \mm )$ is an $\rcd (K,N)$ space and $G \leq \iso (X)$ is a discrete group of isometries with $\diam (X/ G) \leq \varepsilon$, then 
    \begin{equation}\label{eq:rank-bound}
        \rank (G) \leq N,
    \end{equation}
    and in case of equality $X$ is homeomorphic to $\mathbb{R}^N$.  
\end{theorem}

To prove Theorem \ref{thm:contractible} we follow the same strategy Gromov used to prove the Almost Flat Manifolds Theorem \cite{gromov} (see  \cite{buser-karcher} for a much more detailed explanation). After re-scaling, we can assume that $K$ is close to $0$. Relying on results from \cite{breuillard-green-tao, zamora-limits}, we show that if $\varepsilon$ is small enough, then $X$ is Gromov--Hausdorff close to a nilpotent Lie group of dimension $\geq \rank (G)$, proving \eqref{eq:rank-bound}.

In case $\rank (G) = N$, we show that a finite-index subgroup $\Gamma \leq G$ is isomorphic to  a lattice in a simply connected nilpotent Lie group $Y$. Furthermore, a large ball $B $ in $ X $ is both homeomorphic and Gromov--Hausdorff close to a ball in the Euclidean space $\mathbb{R}^N$, the same being true for a ball $B^{\prime}$ in $Y$ when equipped with an appropriate invariant Riemannian metric. Moreover, the homeomorphism $ f : B \to B^{\prime}$ can be taken to be almost equivariant with respect to the local actions of the group $\Gamma$ on both balls. 

Using Lemma \ref{lem:local-to-global-approximation}, the map $f$ is extended to a global $\Gamma$-equivariant homotopy equivalence $\tilde{f} : X \to Y$. Since the Borel Conjecture holds for virtually nilpotent groups, this finishes the proof.  We now give the complete proof.

\begin{proof}
Assuming the theorem fails, there is a sequence  $(X_i , \mathsf{d}_i , \mm _i  )$ of  $\rcd ( K , N ) $ spaces and $G_i \leq \iso (X_i )$ discrete groups of isometries with $\diam (X_i / G_i) \to 0$ such that for all $i$ either $\rank (G_i ) > N $, or $\rank (G_i ) = N$ and $X_i$ is not homeomorphic to $\mathbb{R}^N$. 
After taking a subsequence and re-scaling the spaces $X_i$ by  factors $\lambda_ i$ with $\lambda _i \to \infty$ and $\lambda_i \cdot \diam (X_i / G_i ) \to 0$, we can assume $(X_i, \mathsf{d}_i, \mm_i) $ is an $\rcd (- \frac{1}{i}, N)$ space for each $i$.

For $p_i \in X_i$, by taking again a subsequence and renormalizing the measures $\mm_i$, we can assume $(X_i, \mathsf{d}_i, \mm _i,  p_i)$ converges in the pointed measured Gromov--Hausdorff sense to a pointed $\rcd (K,N)$ space $(X, \mathsf{d}, \mm , p)$. By Lemma \ref{lem:almost-homogeoeus-subgroups}, there is $s  \in \mathbb{N} $ and groups $G_i^{\prime} \leq G_i $ satisfying  \eqref{eq:gi-nilpotent} and \eqref{eq:bounded-index}. By  \cite[Theorem 4.1]{zamora-limits}, $X$ is a nilpotent Lie group equipped with an invariant metric. By Lemma \ref{lem:malcev-construction} and Remark \ref{rem:rank-equality},  for $i$ large enough we have
\begin{equation}\label{eq:rank-inequality-old}
    \rank (G_i ) =  \rank (G_i^{\prime})  \leq   n  \leq N,
\end{equation} 
where $n$ denotes the topological dimension of $X$. This  finishes the proof of \eqref{eq:rank-bound}.

From now on we assume $\rank (G_i ) = N $ for all $i$. Under this assumption,  \eqref{eq:rank-inequality-old} implies that $X$ has topological dimension $N$.  By Theorem \ref{thm:dim}, 
$\mm$ is a multiple of $\mathcal{H}^N$ so it is  $\iso (X)$-invariant, hence by the work of Honda--Nepechiy \cite[Theorem 2]{honda-nepechiy} the metric of $X$ is Riemannian. Moreover, by  \cite[Theorem 2.4]{milnor-invariant}, $X$ is abelian, and by Remark \ref{rem:rank-equality} we have  $(X , p )  = (  \mathbb{R}^N , 0 )$.

 By Corollary \ref{cor:no-small-subgroups}, the groups $G_i^{\prime}$ admit no small subgroups. Then by Lemma \ref{lem:malcev-construction} and Remark \ref{rem:nss-case}, for $i$ large enough there are elements $u_{1,i}, \ldots , u_{n,i} \in \Gamma_i  = G_ i ^{\prime }$, and $N_{1,i}, \ldots , N_{n,i} \in \mathbb{R}^+$ with the following properties:
    \begin{itemize}
        \item The set map $ \alpha _i : \mathbb{Z}^n \to \Gamma  _i$ given by 
        \[  \alpha_i (m_1 , \ldots , m_n) = u_{1,i}^{m_1}\cdots u_{n,i}^{m_n}       \]
        is a bijection.  Moreover, there are polynomials $Q_i : \mathbb {R}^n \times \mathbb{R}^n\to \mathbb{R}^n$ of degree $\leq d(n)$ such that the product in $\Gamma_i$ is given by 
        \[g_1 \cdot g_2 = \alpha _i  ( Q_i ( \alpha _i ^{-1} (g_1) , \alpha _i  ^{-1} (g_2) )  ).  \]\label{property:zamora-1}
        \item There is $C > 0 $ such that the set
        \[  
            P_i : = P (  u_{1,i}, \ldots , u_{n,i} ; N_{1,i}, \ldots , N_{n,i} )  \subset  \Gamma _ i 
        \]  
        is a nilprogression in $C$-normal form with $\thi (P_i ) \to \infty $.  \label{property:zamora-3}
        \item  \label{property:zamora-4} For each $\varepsilon > 0 $ there is $\delta > 0 $ such that 
        \begin{align}
    G( \delta  P_ i  )    \subset  \{ & g \in \Gamma_i \, \vert \, \mathsf{d}_i(g p_i, p_i ) \leq \varepsilon \} ,  \label{eq:grid-small} \\
    \{ g \in  \Gamma_i  \, \vert  \, \mathsf{d}_i( & g p_i , p_i ) \leq \delta \}  \subset     G( \varepsilon  P_i  ) \label{eq:grid-large} 
 \end{align}
 for $i$ large enough.    
    \end{itemize}
By  \eqref{eq:grid-large}, after re-scaling the sequence $X_i$ by a fixed factor, we can assume 
\begin{equation} 
\{ g \in \Gamma _i \vert \mathsf{d} _i (gp_i, p_i) \leq 100 \} \subset G(P_i) . \label{eq:grid-large-2}    
\end{equation}
By  \eqref{eq:grid-small}, there is $\delta_1 > 0 $ with $G(\delta_1 P_i) \subset \{ g \in \Gamma _i \vert \mathsf{d}_i(gp_i,p_i) \leq 1 \} $. Hence there is $C_1 \in \mathbb{N} $ so that 
\begin{equation} \label{eq:grid-small-2}
G(P_i) \subset  G(\delta_1  P_i) ^{C_1} \subset   \{ g \in \Gamma _i \vert \mathsf{d}_i(gp_i, p_i) \leq C_1 \} .   
\end{equation}
Let $\mathfrak{l} _ i $ be the Lie algebra of the Lie group $(\mathbb{R}^n, Q_i )$, and let 
\[ \vv_{j,i} : = \log \left( u_{j,i}^{\lfloor \frac{N_{j,i}}{C}\rfloor }  \right) \in \mathfrak{l}_i  .  \]
By  \eqref{eq:grid-small-2}, after passing to a subsequence, for each $j \in \{ 1, \ldots , n \}$ we can assume
\[    u_{j,i}^{ \lfloor \frac{N_{j,i}}{C}\rfloor  } p_i \to \vv_j \, \, \text{ for some }\vv_j \in X = \mathbb{R}^n.  \]

By Lemmas \ref{lem:almost-translations} and  \ref{lem:euclidean-nilpotent-groups}, after further taking a subsequence, the groups $G_i^{\prime}$ converge equivariantly to the group of translations in $X = \mathbb{R}^n$. It is proven in \cite[Lemma 8.2 and Lemma 2.64]{zamora-limits} that the structure coefficients of $\mathfrak{l}_i$ with respect to the basis $\{ \vv_{1,i}, \ldots , \vv_{n,i}\} $ converge to the structure coefficients in $\mathbb{R}^n$ with respect to $\{ \vv_1, \ldots , \vv_n \}$ as $i \to \infty$. These coefficients are of course $0$ as $\mathbb{R}^n$ is abelian. 

For each $i$, equip $\mathfrak{l}_i$ with the inner product $\langle \cdot , \cdot  \rangle _i  $ that makes the linear isomorphism $\mathfrak{l}_i \to \mathbb{R}^n$ given by $\vv_{j,i} \mapsto \vv_j $ an isometry and let $Y_i$ be the corresponding simply connected nilpotent Lie group equipped with the left invariant Riemannian metric induced from $\langle \cdot , \cdot \rangle _i$. Also let $q_i \in Y_i $ be the identity element and consider the map
\[  f_i :  ( X_i , p_i )  \to  ( Y_i , q_i )    \]
obtained by composing the pointed Gromov--Hausdorff approximation $X_i \to X$ we have by definition, the isometric linear isomorphism  $X \to \mathfrak{l}_i$ that sends $\vv_j \mapsto \vv_{j,i}$, and the exponential map $ \mathfrak{l}_i \to Y_i$. Notice that by Theorem \ref{thm:reifenberg-weak} we can assume $f_i\vert _{B_{100}(p_i)}$ is continuous for $i$ large enough.  

\begin{proposition}\label{pro:almost-equivariant}
For each $\delta  > 0 $, the following holds for $i$ large enough:
\begin{enumerate}[(1)]
    \item $\diam (X_i / \Gamma_i )$, $\diam (Y_i / \Gamma _i) \leq \delta$.\label{pro:almost-equivariant-a}
    \item  For all $x_1, x_2 \in B_{10}(p_i)$,\label{pro:almost-equivariant-b} 
    \[    \vert \textsf{d} ( f_i (x_1), f_i (x_2)) - \mathsf{d}_i(x_1, x_2) \vert \leq \delta .    \]
    \item For all $y \in B_1(q_i)$, there is $x \in B_2 (p_i)$ with 
    \[    \textsf{d} (f_i(x) ,y ) \leq \delta .              \]\label{pro:almost-equivariant-b.5}
    \item If $x \in B_{10}(p_i)$, $g \in \Gamma_i$ are such that $g f_i(x) \in B_9 (q_i)$, then \label{pro:almost-equivariant-c}
    \[     g x \in B_{10}(p_i)  .    \]
    \item If $x \in B_{10}(p_i)$, $g \in \Gamma_i$ are such that $gx \in B_{10}(p_i)$, then \label{pro:almost-equivariant-d}
    \[  d(f_i (gx) , g( f_i (x) )) \leq \delta .     \]
\end{enumerate}

\end{proposition}

\begin{proof}
 \eqref{pro:almost-equivariant-a} is satisfied by hypothesis and the fact that the structure coefficients of $\mathfrak{l}_i$ converge to 0.  \eqref{pro:almost-equivariant-b} and \eqref{pro:almost-equivariant-b.5}  are satisfied as the maps $X_i \to X \to \mathfrak{l}_i \to Y_i$ are each either an isometry or a map uniformly close to being an isometry on a large ball. 

By contradiction, if  \eqref{pro:almost-equivariant-d} doesn't hold, after passing to a subsequence we could find $x_i \in B_{10}(p_i), $ $g_i \in \Gamma_i$ such that $g_i x_i \in B_{10}(p_i)$ and 
\begin{equation}\label{eq:not-almost-equivariant-d}
         \textsf{d} (  f_i (g_i x_i ) , g_i f_i (x_i) ) \geq \delta \, \text{ for all } i.      
\end{equation}
By  \eqref{eq:grid-large-2}, 
\[     g _i   = \exp ( t_{1,i} \vv_{1,i} ) \cdots \exp ( t_{n,i} \vv_{n,i}  )              \]
for some $t_{1,i}, \ldots , t_{n,i} \in [ -1 , 1 ] $.  After passing again to a subsequence  we have $t_{j,i} \to t_j$ as $i \to \infty$ for some $t_j \in [-1,1]$. Then by how we defined the vectors $\vv_j$, it follows that $\exp ( t_{j,i} \vv _ {j,i})$ converges to the translation in $X$ by the vector $t_j \vv_j$. As the convergence is equivariant, $g_i$ converges to the translation by $ t_1 \vv _ 1 + \ldots + t_n \vv_n$. 

On the other hand, by looking at the pointed Gromov--Hausdorff approximations $Y_i \to X$ obtained by composing $\log : Y_i \to \mathfrak{l}_i$ with the linear isomorphism $\mathfrak{l}_i \to X$ that sends $\vv_{j,i} \mapsto \vv_j$, the sequence of actions of $\Gamma _i $ on $Y_i$ converge (again by Lemmas \ref{lem:almost-translations} and \ref{lem:euclidean-nilpotent-groups}) to the group of translations of $X$. With respect to this convergence, by how we chose the map $\mathfrak{l}_i \to X$, the sequence $\exp (t_{j,i}\vv_{j,i})$ converges again to the translation in $X$ by $t_j \vv_j$, so $g_i$ again converges to $t_1\vv_1 + \ldots + t_n \vv_n $. 

The action of $g_i$ on both $X_i$ and $Y_i$ converges to the same isometry of $X$, and the pointed Gromov--Hausdorff approximations $f_i : X_i \to Y_i $ were obtained by composing the ones we used for $X_i \to X$ with the approximate inverses of the ones we used for $Y_i \to X$, meaning that $f_i$ approximately commutes with the action of $g_i$, contradicting  \eqref{eq:not-almost-equivariant-d}.

Lastly, assume  \eqref{pro:almost-equivariant-c} fails, so after taking a subsequence there are $x_i \in B_{10}(p_i)$, $g_i \in \Gamma_i$ with $g_i f_i (x_i)  \in B_{9}(q_i) $ but $ g_i x_i  \not\in  B_{10}(p_i) $ for all $i$. After taking again a subsequence both sequences  $x_i$ and $f_i (x_i)$ converge to a point $x \in X$ and $g_i$ converges to an isometry of $X$. On one hand, since $g_i f_i (x_i) \in B_9(q_i)$ for all $i$, we have $gx \in B_{19/2} (p)$, but on the other hand since $g_i x_i \not\in B_{10}(p_i)$ for all $i$ we also have $gx \not\in B _{19/2}(p)$. This is a contradiction.
\end{proof}

Now fix $\varepsilon  > 0 $ given by Lemma \ref{lem:transfer-trivial-homology} with $m = n $. By Proposition \ref{pro:almost-equivariant}, the conditions of Lemma \ref{lem:local-to-global-approximation} hold for $i$ large enough so we have continuous $\Gamma_i$-equivariant maps $\tilde{f}_i : X_i \to Y_i $ such that 
 \begin{enumerate}[(i)]
     \item For all $x_1, x_2 \in X_i$ with $\min \{ \mathsf{d}_i(x_1, x_2) ,  d(\tilde{f}_i (x_1),\tilde{f}_i (x_2)) \}  \leq 2 $  one has 
     \[   \vert  \textsf{d}  ( \tilde {f}_i (x_1) \tilde{f} _i (x_2) ) - \mathsf{d}_i(x_1, x_2) \vert \leq \varepsilon /3.     \] \label{conclusion:f-i-tilde-1}
     \item  For all $y \in Y_i$, there is $x \in X_i$ with $ \textsf{d}  ( \tilde{f} _i x , y ) \leq \varepsilon / 3$.\label{conclusion:f-i-tilde-2}
 \end{enumerate}
By \eqref{conclusion:f-i-tilde-2}, we can define $h_i : Y_i \to  X_i $ so that $  \textsf{d}  ( \tilde{f}_i h_ i (y) , y ) \leq \varepsilon / 3 $ for all $y \in Y_i$. Then from \eqref{conclusion:f-i-tilde-1}, follows that
\begin{itemize}
    \item For all $y_1, y_2 \in Y_i$ with $ \textsf{d} (y_1, y_2) \leq 1$, one has $\mathsf{d}_i( h_i y_1, h_i y_2 ) \leq  \textsf{d} ( y_1 , y_2 ) + \varepsilon $.
    \item  For all $x \in X_i$, one has $\mathsf{d}_i(h_i(\tilde{f}_i(x)), x) \leq \varepsilon$.
\end{itemize}
By Corollary \ref{cor:locally-contractible-reifenberg}, for $i$ large enough the following holds.
\begin{itemize}
    \item For each $r \in ( 0 , 1 ]$, $x \in X_i$, the ball $B_r(x)$ is contractible in $B_{r + \varepsilon }(x )$.
\end{itemize}
This means that all the conditions of Lemmas \ref{lem:transfer-trivial-homology} and  \ref{lem:transfer-simply-connected} are satisfied, hence the spaces $X_i$ are simply connected and have trivial homology up to degree $n$. By Corollary \ref{cor:manifold-reifenberg} the spaces $X_i$ are manifolds of dimension $n$ so from Lemma \ref{lem:contractible} they are contractible. Since the groups $\Gamma _i$ are discrete and torsion-free, they act freely and the quotient $X_i / \Gamma _ i $ is a closed aspherical manifold. By Theorem \ref{thm:borel-nilpotent}, $X_i / \Gamma _ i $ is homeomorphic to a nilmanifold, so its universal cover $X_i$ is homeomorphic to $\mathbb{R}^n= \mathbb{R}^N$.
\end{proof}
\begin{proof}[Proof of Theorem \ref{thm:max-rank-nilmanifold}]
    By Theorem \ref{thm:contractible} applied to the universal covers, there is $\varepsilon (K,N) > 0$ such that if $\diam (X) \leq \varepsilon $ for an $\rcd (K,N)$ space $X$, then $\rank (\pi_1(X)) \leq N$, and in case of equality $X$ is an aspherical manifold. From the proof of Theorem \ref{thm:contractible} we deduce that $\pi_1(X)$ is virtually nilpotent, so the result follows from Theorem \ref{thm:borel-nilpotent}. 
\end{proof}

\section{Eigenmap to a flat torus}\label{sec:canonical}

Before we jump into the proof of Theorem \ref{thm:equivariant-torus}, we make a few remarks. If $(X,\mathsf{d}, \mm )$ is an $\rcd (K,N)$ space, and $G \leq \iso (X)$ a group of measure preserving isometries, then $G$ admits a natural isometric linear action on $\sobo (X)$ by 
\[   g^{\ast} f (x) : = f (g ^{-1} x ) , \, \,\,\, g \in G , f \in \sobo (X), x \in X .    \]
This action clearly commutes with the Laplacian. That is, if  $f \in \sobo (X)$ is in the domain of the Laplacian, then $g^{\ast}f $ is also in the domain of the Laplacian, and
\begin{equation*}
    \Delta (g^{\ast} f) = g^{\ast} \Delta f  .    
\end{equation*}
In particular, for each $\lambda \geq 0$, the eigenspace
\[     \{  f \in \sobo (X) \vert \Delta f = - \lambda f  \}    \]
is a $G$-invariant subspace of $\sobo (X)$.

\begin{proof}[Proof of Theorem \ref{thm:equivariant-torus}]

Assume by contradiction there is sequence of $\rcd (K,N)$ spaces $(X_i, \mathsf{d}_i, \mm _i )$ that converges in the measured Gromov--Hausdorff sense to $\mathbb{T}^N$ and there are abelian groups $G_i \leq \iso (X_i)$ with $\diam (X_i / G_i ) \to 0$,   but the thesis of the theorem fails for each $X_i$.

By Lemma \ref{lem:almost-translations}, the sequence $G_i$ converges equivariantly to a connected abelian group $G \leq \iso (\mathbb{T}^N)$ acting freely and transitively, which can be naturally identified with  $\mathbb{T}^N$. It is well known that the spectrum of the Laplacian on $\mathbb{T}^N$ consists of non-positive integers, the multiplicity of the eigenvalue $-1$ is $2N$, and the eigenfunctions corresponding to that eigenvalue are given by the coordinate functions. That is, if $\varphi_j, \psi_j : \mathbb{T}^N \to \mathbb{R}$ are given by 
\begin{equation}\label{eq:torus-inclusion}
    x = (  \varphi_1 (x) , \psi_1 (x), \ldots , \varphi_N (x) , \psi _N (x) ) \, \text{ for all } x \in \mathbb{T}^N  , 
\end{equation}  
then
\[   \Delta \varphi_j = - \varphi_j  , \,  \Delta \psi _j = - \psi _j \, \text{ for each } j \in \{ 1, \ldots , N \} ,   \]
and $ \{ \varphi_j, \psi _j \}_{j =1}^N $ is a set of orthogonal functions in $H^{1,2}(\mathbb{T}^N )$ with the same norm. Set 
\[   W_ j : = \langle \varphi_j \rangle \oplus \langle \psi_j\rangle   = \langle G \varphi _j \rangle  , \hspace{3cm} V : = \bigoplus_{j = 1 }^N W_j .  \hspace{1.2cm}        \]
Then it is not hard to check that for $f \in V$, one has
\begin{equation}\label{eq:f-dim-3}
    f \in W_j \text{ for some } j \hspace{0.5cm} \Longleftrightarrow \hspace{0.5cm} \dim _{\mathbb{R}}(\langle G f \rangle ) \leq 2 . \hspace{1.3cm}
\end{equation}
For each $i \in \mathbb{N}$, set 
\[   V_i : = \bigoplus_{1/2 < \lambda < 3/2} \{ f \in \sobo (X_i)  \vert \Delta f = - \lambda f .  \}     \]
By Corollary \ref{cor:spectrum-continuity}, for $i$ large enough, $\dim _{\R} (V_i )  = 2N$, and by Lemma \ref{lem:simultaneous-diagonalization},
\begin{equation}\label{eq:vi-decomposition}
     V_{i} = \bigoplus_{j=1}^N W_{j,i}    
\end{equation}
with each $W_{j,i}$ 2-dimensional, $G_i$-invariant, and spanned by eigenfunctions. 

\begin{lemma}\label{lem:w-ji-to-w-j}
    Assume for some choice of $j ( i ) \in \{ 1, \ldots , N \} $, a sequence of functions  $f_i \in W_{j(i), i}$  converges in $\sobo$ to some function $f \in \sobo (\mathbb{T}^N)$. Then $f \in W_{j_0}$ for some $j_0 \in \{ 1 , \ldots , N \}$. 
\end{lemma}

\begin{proof}
   By  \eqref{eq:f-dim-3}, if the lemma fails, there are $n_k\in \mathbb{N}$ for $k \in \{ 1, 2, 3 \}$, $\alpha _{k, \ell } \in \mathbb{R}$, $g_{k, \ell} \in G$ for $ \ell \in \{ 1, \ldots , n_k \}$, such that the vectors 
    \[    \vv _k : = \sum_{\ell = 1}^{n_k} \alpha_{k, \ell} g_{k,\ell }^{\, \ast}(f) \in V   \]
 with  $k \in \{ 1, 2, 3 \}$ form an orthonormal set. Take elements $g_{k,\ell, i } \in G_i$  with $g_{k, \ell , i } \to g_{k, \ell}$ as $i \to \infty$. Then for fixed $k$, the vectors
 \[     \vv_{k, i} : =  \sum_{\ell = 1}^{n_k} \alpha_{k, \ell} g_{k,\ell,i }^{\, \ast}(f_i) \in W_{j(i),i}                  \]
converge to $\vv_k$ in $\sobo$. Hence as $i \to \infty$, for $k \neq \ell$ one has
\[  \Vert \vv_{k,i} \Vert_{\sobo} \to 1  ,\hspace{3cm} \langle \vv_{k,i} , \vv_{\ell, i} \rangle_{\sobo} \to 0 .   \]
This is impossible as $\dim _{\R}(W_{j(i), i}) = 2$.
\end{proof}

By Theorem \ref{thm:spectrum-continuity}\eqref{thm:spectrum-continuity-2}, for each $j \in \{ 1, \ldots , N \} $ there are eigenfunctions $\varphi_{j,i} \in V_i$ with $\varphi_{j,i} \to \varphi_j$ uniformly and in $\sobo$. 

\begin{lemma}\label{lem:min-d-phi-w}
    For each $j \in \{ 1, \ldots , N \}$, the quantity 
    \begin{equation}\label{eq:min-d-phi-w}
        \min  \{  \Vert \varphi_{j,i} - \vv \Vert _{\sobo } \vert \,  \vv \in W_{\ell, i} , \ell \in \{ 1, \ldots , N \}  \} 
    \end{equation}    
    goes to $0$ as $i \to \infty$.
\end{lemma}
\begin{proof}
    If the lemma fails, after passing to a subsequence, there is a sequence $\ell (i) \in \{ 1, \ldots , N \} $ for which we can write 
    \[   \varphi_{j,i} =  \varphi_{j,i}^a + \varphi_{j,i}^b                              \]
    with $\varphi_{j,i}^a \in W_{\ell (i) ,i}$, $\varphi_{j,i}^b \in W_{\ell (i), i}^{\perp} \leq V_i$, and 
     \[ \min \left\{   \liminf_{i \to \infty}    \Vert \varphi_{j,i} ^a \Vert_{\sobo} , \liminf_{i \to \infty}  \Vert \varphi_{j,i}^b \Vert_{\sobo } \right\}  > 0 .  \]
     This is done by choosing $\ell (i) \in \{ 1, \ldots, N \} $ that attains the minimum in  \eqref{eq:min-d-phi-w}, and defining  $\varphi_{j,i}^a$ to be the orthogonal projection of $\varphi_{j,i}$ to $W_{\ell (i),i}$. Also let $\ww _ i \in W_{\ell (i), i} $ be a unit vector orthogonal to $\varphi_{j,i}^a$. Since each $\varphi_{j,i}^a$, $\varphi_{j,i}^b$, $\ww_i$ is a linear combination of eigenfunctions in $V_i$, by Theorem \ref{thm:spectrum-continuity}\eqref{thm:spectrum-continuity-1} after passing to a subsequence we can assume that 
     \[     \varphi_{j,i}^a \to \vv^a ,  \hspace{1cm} \varphi_{j,i}^b \to \vv^b , \hspace{1cm} \ww _i \to \ww     \]
     uniformly and in $\sobo$ for some mutually orthogonal non-zero vectors $\vv^a, \vv^b, \ww \in V$.  By Lemma \ref{lem:w-ji-to-w-j}, $\{ \vv ^a , \ww \}$ forms a basis of $W_{j_0}$ for some $j_0$, and $\vv ^b \in W_{j_0} ^{\perp} \leq V$. This would mean that $\varphi_j = \vv ^a + \vv ^b + \ww$ is not in  $W_{j}$; a contradiction. 
\end{proof}

Let $\varphi_{j,i}^a$, $\ww _i$ be as in the proof of Lemma \ref{lem:min-d-phi-w} and let $\psi _{j,i } ^a : = \Vert \varphi_{j,i} ^a \Vert _{\sobo} \ww _i \in W_{\ell (i) , i } $. Then
\[      \varphi_{j,i}^a \to \varphi_j , \hspace{3cm} \psi _{j,i}^a \to \pm \psi_{j}       \]
uniformly and in $\sobo$ as $i \to \infty$. After possibly multiplying $\psi_{j,i}^a$ by $-1$, we can assume $\psi_{j,i}^a \to \psi_j$. We point out that the awkwardness of $\varphi_{j,i}^a$ being in $W_{\ell, i}$ instead of $W_{j,i}$ emanates from the fact that we have a repeated eigenvalue in the limit space, and we have a priori  no way of telling apart the spaces $W_{j,i}$ (see the discussion before \cite[Theorem 7.3]{cheeger-colding-iii}). 

Consider the map $\Phi_i : X_i \to V_i$ given by 
\[   \Phi_i (x)  = \sum_{j = 1}^N [  \varphi_{j,i} ^a (x) \varphi_{j,i} ^a + \psi _{j,i}^a (x) \psi _{j,i}^a ] .   \]
Fix $g \in G_i$. Since $W_{j,i}$ is $G_i$-invariant for each $j$, one has 
\begin{align*}
g^{\ast} \varphi_{j,i}^a & = \cos ( \theta _{j,i} )\varphi_{j,i}^a + \sin  (\theta _{j,i} )\psi _{j,i}^a ,  \\
g^{\ast} \psi_{j,i}^a & = - \sin ( \theta _{j,i} )\varphi_{j,i}^a + \cos  (\theta _{j,i} )\psi _{j,i}^a 
\end{align*}
for some angles $\theta_{j,i} \in \mathbb{S}^1$. Hence
\begin{align*}
  g ^{\ast} \Phi_i (x)    & = \sum _{j =1} ^N [ g^{\ast} \varphi_{j,i} ^a (g  x)  g^{\ast}\varphi_{j,i} ^a + g^{\ast}\psi _{j,i}^a ( g x) g^{\ast} \psi _{j,i}^a ]  \\
  & = \sum _{j = 1}^N  [ \cos ^2  (\theta _{j,i} ) + \sin ^2 (\theta _{j,i} )] [  \varphi_{j,i} ^a (g  x) \varphi_{j,i} ^a + \psi _{j,i} ^a (gx) \psi _{j,i} ^a ] \\
  & = \Phi_i (gx),
\end{align*}
so $\Phi_i$ is $G_i$-equivariant. For each $i$, let $U_i \subset V_i$ denote the set of vectors whose projection to $W_{j,i}$ is non-zero for all $j$, and let $\pi_i : U_i \to U_i $ be the map given by 
\[   \pi _i (z_1, \ldots , z_N) : = \left(  \frac{z_1}{\Vert z_1 \Vert _{\sobo} } , \ldots , \frac{z_N}{\Vert z_N \Vert _{\sobo} }  \right)                \]
with $z_j \in W_{j,i} $ for each $j$. Since the decomposition in  \eqref{eq:vi-decomposition} is $G_i$-invariant, then so is the set $U_i$, and $\pi_i$ commutes with the action of $G_i$. Finally, we identify $V$ and $V_i$ with $\mathbb{R}^{2N}$ via the bases $\{ \varphi_1, \psi _1 , \ldots , \varphi_N , \psi _N \}  $ and $ \{ \varphi _{1,i}^a , \psi _{1,i}^a, \ldots , \varphi _{N,i}^a , \psi _{N,i}^a \} $ respectively. After this identification, $\Phi_i$ converges uniformly to the isometric inclusion of $\mathbb{T}^N$ into $\mathbb{R}^{2N}$ given by  \eqref{eq:torus-inclusion}. By Corollary \ref{cor:no-small-subgroups} combined with Corollary \ref{cor:free},  the action of $G_i$ on $X_i$ is free for large $i$. Then by Theorem \ref{thm:canonical-homeomorphism} combined with Remark \ref{rem:equi-regular}  and the fact that both $\Phi_i$ and $\pi_i$ are $G_i$-equivariant, the result follows.  
\end{proof}

\begin{proof}[Proof of Theorem \ref{thm:mmp}] By contradiction, assume there is a sequence of $\rcd (K,N)$ spaces $(X_i, \mathsf{d}_i, \mm_i)$ with 
\[ \diam (X_i) \to 0,  \hspace{3cm}  \bet (X_i) = N , \]
but $X_i$ not bi-H\"older homeomorphic to a flat torus. In \cite{mondello-mondino-perales}, it is shown that for $i$ large, $X_i$ is homeomorphic to $\mathbb{T}^N$, and a sequence of finite sheeted covers $X_i^{\prime} \to X_i$ is constructed in such a way that $X_i^{\prime}$ converges in the measured Gromov--Hausdorff sense to the flat torus $ \mathbb{T}^N$. 

Let $G_i \leq \iso (X_i')$ be the group of deck transformations. By Theorem \ref{thm:equivariant-torus}, for $i$ large enough, there are free actions $G_i \to \iso (\mathbb{T}^N)$ and $G_i$-equivariant bi-H\"older homeomorphisms $F_i : X_i ' \to \mathbb{T}^N$ that descend to bi-H\"older homeomorphisms $X_i = X_i ' / G_i \to \mathbb{T}^N/G_i$. Since $\mathbb{T}^N/G_i$ is locally isometric to $\mathbb{T}^N$, it is also a flat torus, which is a contradiction.
\end{proof}

\subsection*{Acknowledgments}
The authors would like to thank Shouhei Honda, Raquel Perales, Xiaochun Rong, and Daniele Semola for helpful discussions, Vitali Kapovitch for sharing the reference \cite{fisher-kalinin-spazier}, and Christine Escher for carefully reading and providing comments on a previous version of this paper. The authors are also grateful to two anonymous referees for their valuable feedback.  While preparing this paper, Sergio Zamora held a Postdoctoral Fellowship at Max Planck Institute for Mathematics at Bonn. Xingyu Zhu held a Postdoctoral researcher position at the University of Bonn and gratefully acknowledge the financial support by the Deutsche Forschungsgemeinschaft (DFG) within the CRC 1060, at University of Bonn project number 211504053.


\end{document}